\documentclass[10pt]{article}
\pdfoutput=1 
\usepackage{latexsym}
\usepackage{amsmath,amsthm}
\usepackage{amssymb,amsfonts}
\usepackage{mathrsfs}
\usepackage{epsfig}
\usepackage{color}

\def\N{\mathbb{N}}

\def\R{\mathbb{R}}

\def\J{\mathscr{J}}
\def\F{\mathscr{F}}
\def\T{\mathscr{T}}
\def\G{\mathscr{G}}
\def\H{\mathscr{H}}
\def\na{\nabla}
\def\pa{\partial}
\def\RR{\mathbb{R}}
\def\eps{\varepsilon}
\def\vphi{\varphi}
\def\A{\mathscr A}

\def\p{{\bf p}}
\def\q{{\bf q}}

\def\ds{\displaystyle} 
\def\div{{\rm div}}
\def\Lip{{\rm Lip}}
\def\refe#1{(\ref{#1})}

\newtheorem{theorem}{Theorem}[section]
\newtheorem{lemma}[theorem]{Lemma}
\newtheorem{definition}[theorem]{Definition}
\newtheorem{remark}[theorem]{Remark}
\newtheorem{proposition}[theorem]{Proposition}
\newtheorem{corollary}[theorem]{Corollary}

\author{Antoine Mellet\thanks{Department of Mathematics, University of Maryland, College Park MD 20742, USA. Email: mellet@math.umd.edu}{ }\thanks{Partially supported by NSF Grant DMS-0901340} and Julien Vovelle\thanks{Universit\'e de Lyon ; CNRS ; Universit\'e Lyon 1, Institut Camille Jordan,  43 boulevard du 11 novembre 1918, F-69622 Villeurbanne Cedex, France. Email: vovelle@math.univ-lyon1.fr}}

\date{}
\title{Existence and regularity of extremal solutions for a  mean-curvature equation}

\begin{document}
\maketitle

\begin{abstract}
We study a class of mean curvature equations $-\mathcal Mu=H+\lambda u^p$
where $\mathcal M$ denotes the mean curvature operator and for $p\geq 1$. 
We  show that there exists an extremal parameter $\lambda^*$ such that this equation admits a minimal weak solutions for all $\lambda \in [0,\lambda^*]$, while no weak solutions exists for $\lambda >\lambda^*$ (weak solutions will be defined as critical points of a suitable functional).
In the radially symmetric case, we then show that minimal weak solutions are classical solutions for all $\lambda\in [0,\lambda^*]$ and that another branch of classical solutions exists in a neighborhood $(\lambda_*-\eta,\lambda^*)$ of $\lambda^*$.
\end{abstract}

\vspace{10pt}

\noindent{\it 2000 Mathematics Subject Classification:}  53A10 (35J60)
\vspace{10pt}

\noindent{\it Keywords:} Mean curvature, minimal solution, semi-stable solution, extremal solution, regularity. 
\vspace{10pt}


\section{Introduction}


Let $\Omega$ be a bounded open subset of $\RR^n$ with smooth boundary $\pa\Omega$. 
The aim of this paper is to study the existence and regularity of non-negative solutions for the following mean-curvature problem:
\begin{equation}\label{eq:1} \tag{\mbox{$P_\lambda$}}
\left\{
\begin{array}{ll}
-\mbox{div} (Tu) = H+\lambda f(u) &\mbox{ in } \Omega,\\[5pt]
u=0 & \mbox{ on } \pa\Omega ,
\end{array}
\right.
\end{equation}
where 
$$ Tu: =\frac{\na u}{\sqrt{1+|\na u|^2}}$$
and
$$ f(u) = |u|^{p-1} u,\qquad p\geq 1.$$ 
Formally, Equation \refe{eq:1} is the Euler-Lagrange equation for the minimization of the functional
\begin{equation}\label{eq:F}
 \F_\lambda(u) := \int_\Omega \sqrt{1+|\na u|^2} - \int_\Omega  H u + \lambda F(u)\, dx + \int _{\pa \Omega} |u|\, d\H^{n-1}(x)
\end{equation}
with $F(u) = \frac{1}{p+1} |u|^{p+1}$ (convex function).
\smallskip

When $\lambda=0$, Problem (\ref{eq:1}) reduces to a prescribed mean-curvature equation, which has been extensively studied (see for instance  Bernstein \cite{Bernstein10}, Finn \cite{Finn65}, Giaquinta \cite{Giaquinta74}, Massari \cite{Ma} or Giusti \cite{Giusti76,Giusti78}).
In particular, it is well known that a necessary condition for the existence of a classical solution of (\ref{eq:1}) when $\lambda=0$ (or the existence of a minimizer of $ \F_{\lambda=0}$) is
\begin{equation}\label{eq:A}
\left| \int_A H\, dx\right| < P(A), \mbox{ for all proper subset $A$ of $\Omega$,} 
\end{equation}
where $P(A)$ is the perimeter of $A$ (see \refe{defperimeter} for the definition of the perimeter).
It is also known that the following is a sufficient condition (see Giaquinta \cite{Giaquinta74}):
\begin{equation}\label{eq:AA}
\left| \int_A H\, dx\right| \leq (1-\eps_0) P(A), \mbox{ for all measurable set $A\subset\Omega$,} 
\end{equation}
for some $\eps_0>0$.
\smallskip


Equation \refe{eq:1} has also been studied for $\lambda<0$ and $p=1$ ($f(u)=u$), in particular in the framework of capillary surfaces (in that case, the Dirichlet boundary condition is often replaced by a  contact angle condition. We refer the reader to the excellent book of Finn \cite{Finn86} for more details on this topic).
The existence of minimizers of (\ref{eq:F}) when $\lambda<0$ is proved, for instance, by  Giusti \cite{Giusti76} and Miranda \cite{Mi}.
\smallskip

In this paper, we are interested in the case $\lambda>0$. In that case, the functional $ \F_\lambda$ is not convex, and the existence and regularity results that hold when $\lambda\leq 0$ no longer apply. The particular case $p=1$ corresponds to the classical pendent drop problem (with the gravity pointing upward in our model).
The pendent drop in a capillary tube (Equation (\ref{eq:1}) in a fixed domain but with contact angle condition rather than Dirichlet condition) has been studied in particular by  Huisken \cite{Huisken83}-\cite{Huisken84}, while 
the corresponding free boundary problem, which describes a pendent drop hanging from a flat surface has been studied by  Gonzalez,  Massari and  Tamanini \cite{GMT} and  Giusti \cite{Giusti80}.
In \cite{Huisken84}, Huisken also  studies the Dirichlet boundary problem (\ref{eq:1}) when $p=1$ (with possibly non-homogeneous boundary condition). This problem models a pendent drops hanging from a fixed boundary, such as the end of a pipette. Establishing suitable gradient estimates, Huisken proves the existence of a solution for small $\lambda$ (see also Stone \cite{stone} for a proof by convergence of a suitable evolution process).
In \cite{CF}, Concus and Finn characterize the profile of the radially symmetric pendent drops, thus finding explicit solutions for this mean curvature problem.
Finally, in the case $H=0$, other power like functions $f(u)$ have been considered, in particular by  Pucci and  Serrin \cite{ps86} and  Bidaut-V\'eron \cite{MFBV}. In that case, non-existence results can be obtained for $f(u)=u^p$ if $p\geq \frac{N+2}{N-2}$. 
Note however that in our paper, we will always assume that $H>0$ (see condition (\ref{eq:CC})), and we will in particular show that a solution exists for all values of $p$, at least for small $\lambda>0$.

\subsection{Branches of minimal and non-minimal weak solutions}
Through most of the paper, we will study {\it weak} solutions of (\ref{eq:1}), which we will define as critical points of a suitable functional  in $\mathrm{BV}(\Omega)\cap L^{p+1}(\Omega)$ (see Definition \ref{defsol}).
In the radially symmetric case, we will see that those weak solutions are actually {\it classical} solutions (see Section~\ref{sec:classicalsol}) in $\mathcal C^{2,\alpha}(\overline{\Omega})$ of (\ref{eq:1}). 

As noted above, a first difficulty when $\lambda>0$ is that the functional $\F_\lambda$ is not convex and not bounded below. So global minimizers clearly do not exist. 
However, under certain assumptions on $H$ (which guarantee the existence of a solution for $\lambda=0$), it is still possible to show that solutions of (\ref{eq:1}) exist for small values of $\lambda$ (this is proved in particular by Huisken \cite{Huisken84} in the case $p=1$).
The goal of this paper is to show, under appropriate assumptions on $H$ and for $p\geq 1$ that
\begin{itemize}
\item[1.] there exists an extremal parameter $\lambda^*>0$ such that (\ref{eq:1}) admits a minimal non-negative weak solutions $u_\lambda$ for all $\lambda \in [0,\lambda^*]$, while no weak solutions exists for $\lambda >\lambda^*$ (weak solutions will be defined as critical points of the energy functional that satisfy the boundary condition (see Definition~\ref{defsol}), and by minimal solution, we mean the smallest non-negative solution),
\item[2.] minimal weak solutions are uniformly bounded in $L^\infty$ by a constant depending only on $\Omega$ and the dimension. 
\end{itemize}
We then investigate the regularity of the minimal weak solutions, and prove that
\begin{itemize}
\item[3.] in the radially symmetric case, the set $\{u_\lambda\,;\, 0\leq \lambda\leq \lambda^*\}$ is a branch of classical solutions
(see Section~\ref{sec:classicalsol} for a precise definition of classical solution).
In particular, we will show that the extremal solution $u_{\lambda^*}$, which is the increasing limit of $u_\lambda$ as $\lambda\to\lambda^*$, is itself a classical solution. 
\item[4.] It follows that in the radially symmetric case, there exists another branch of (non-minimal) solutions for $\lambda$ in a neighborhood $[\lambda^*-\eta,\lambda^*]$ of $\lambda^*$.
\end{itemize}

Those results will be stated more precisely in Section~\ref{sec:mainresults}, after we introduce some notations in Sections \ref{sec:weaksolu} and \ref{sec:classicalsol}.
The rest of the paper will be devoted to the proofs of those results.

\subsection{Semi-linear elliptic equations}

These results and our analysis of Problem~\refe{eq:1} are guided by the study of the following classical problem:
\begin{equation}\label{eq:laplace}
\left\{
\begin{array}{ll}
\ds -\Delta u = g_\lambda(u) & \mbox{ in } \Omega \\[2pt]
\ds u=0& \mbox{ on } \pa\Omega.
\end{array}
\right.
\end{equation}
It is well known that if $g_\lambda(u)=\lambda f(u)$, with $f$ superlinear and $f(0)>0$, then 
there exists a critical value $\lambda^*\in(0,\infty)$ for the parameter $\lambda$  such that one (or more) solution exists for $\lambda <\lambda^*$, a unique weak solution $u^*$ exists for $\lambda=\lambda^*$ and there is no solution for $\lambda>\lambda^*$ (see \cite{CR}).
And one of the key issue in the study of \refe{eq:laplace} is whether the extremal solution $u^*$ is a classical solution or $u_\lambda$ blows up when $\lambda\to\lambda^*$ (see \cite{KK,BC,MironescuRadulescu96,martel}).

Classical examples that have been extensively studied include power growth $g_\lambda (u)=\lambda(1+u)^p$ ($p>1$) and the celebrated Gelfand problem $g_\lambda(u) = \lambda e^u$ (see \cite{joseph-lundgren,mignot-puel,BV}).
For such non-linearities, the minimal solutions, including the extremal solution $u^*$ can be proved to be classical, at least in low dimension.
In particular, for $g_\lambda (u)=\lambda(1+u)^p$, $u^*$ is a classical solution if 
$$n-2 <F(p):= \frac{4p}{p-1} +4 \sqrt{\frac{p}{p-1}}$$
(see Mignot-Puel \cite{mignot-puel}) while when $\Omega=B_1$ and $n-2\geq F(p)$, it can be proved that $u^*\sim Cr^{-2}$ (see Brezis-V\'azquez \cite{BV}).
For very general non-linearities of the form $g_\lambda (u)=\lambda f(u)$ with $f$ superlinear, Nedev \cite{nedev} proves the regularity of $u^*$ in low dimension while Cabr\'e \cite{Cabre06} and  Cabr\'e-Capella \cite{CabreCapella06,CabreCapella07} obtain optimal regularity results for $u^*$ in the radially symmetric case.

Other examples of non-linearity have been studied, such as $g_\lambda(x,u)=f_0(x,u)+\lambda \varphi(x)+f_1(x)$ (see Berestycki-Lions \cite{berestyckilions81}) or $g_\lambda(x,u)= \lambda f(x)/(1-u)^2$ (see Ghoussoub et al.~\cite{GhoussoubGuo06,EspositoGhoussoubGuo07,GhoussoubGuo08}). 
\vspace{10pt}

Our goal is to study similar behavior for the mean-curvature operator.
In the present paper, we only consider functions $g_\lambda (u)=H+\lambda u^p$, but the techniques introduced here can and will be extended to more general non-linearities in a forthcoming paper.


\section{Definitions and main theorems}


\subsection{Weak solutions}\label{sec:weaksolu}

We recall that $\mathrm{BV}(\Omega)$ denotes the set of functions in $L^1(\Omega)$ with bounded total variation over $\Omega$, that is:
\begin{equation*}
\ds\int_\Omega|Du|:=\sup\left\{\ds\int_\Omega u(x)\div(g)(x)\, dx;g\in\mathcal C^1_c(\Omega)^n,|g(x)|\leq 1\right\} <+\infty.
\end{equation*}
The space $\mathrm{BV}(\Omega)$ is equipped with the norm
$$ \|u\|_{ \mathrm{BV}(\Omega)} = \|u\|_{L^1(\Omega)} + \int_\Omega|Du|.$$
If $A$ is a Lebesgue subset of $\RR^n$, its perimeter $P(A)$ is defined as the total variation of its characteristic function $\varphi_A$:
\begin{equation}\label{defperimeter}
P(A):=\ds\int_{\RR^n}|D\varphi_A|,\quad \varphi_A(x)=\left\{\begin{array}{ll}1 &\mbox{ if }x\in A,\\ 0&\mbox{ otherwise.}\end{array}\right. 
\end{equation}
For $u\in \mathrm{BV}(\Omega)$, we define the ``area" of the graph of $u$ by
\begin{equation}\label{defLu}
\ds\A(u):=\int_\Omega \sqrt{1+|\na u|^2}=\sup\left\{\ds\int_\Omega g_0(x)+u(x)\div(g)(x)\, dx\right\},
\end{equation}
where the supremum is taken over all functions $g_0\in\mathcal  C^1_c(\Omega), g\in\mathcal  C^1_c(\Omega)^n$ such that $|g_0|+|g|\leq 1$ in $\Omega$. An alternative definition is $\A(u)=\int_{\Omega\times\R}|D\varphi_U|$ where $U$ is the subgraph of $u$.
We have, in particular
\begin{equation}
\int_\Omega |Du| \leq \int_\Omega \sqrt{1+|\na u|^2} \leq |\Omega| + \int_\Omega |Du|.
\label{eq:AreaVariation}\end{equation}

\vspace{10pt}

A major difficulty, when developing a variational approach to (\ref{eq:1}), is to deal with the boundary condition. It is well known that even when $\lambda=0$, minimizers of $\F_\lambda$ may not satisfy the homogeneous Dirichlet condition (we need an additional condition on $H$ and the curvature of $\pa\Omega$, see below condition~\refe{eq:B}).
Furthermore, the usual techniques used to handle this issue, which work when $\lambda\leq 0$  do not seem to generalize easily to the case $\lambda>0$.
For this reason, we will not use the functional $\F_\lambda$ in our analysis. Instead, we will define the solutions of (\ref{eq:1}) as the ``critical points" (the definition is made precise below, see Definition~\ref{defsol} and Remark~\ref{rk:defsol}) of the functional
\begin{equation}\label{eq:J}
\J_\lambda(u): = \int_\Omega \sqrt{1+|\na u|^2} - \int_\Omega H(x) u + \lambda F(u)\, dx 
\end{equation}
which satisfy the boundary condition $u=0$ on $\pa\Omega$.

\begin{proposition}[Directional derivative of the area functional] For any $u,\varphi\in \mathrm{BV}(\Omega)$ the limit
\begin{equation}
{\mathcal L} (u)(\varphi):=\ds\lim_{t\downarrow 0}\ds\frac{1}{t}\left(\A(u+t\varphi)-\A(u)\right)
\label{def:Lu}\end{equation}
exists and, for all $u,v\in \mathrm{BV}(\Omega)$
\begin{equation}
\A(u)+{\mathcal L} (u)(v-u)\leq\A(v).
\label{eq:AconvexOrdre1}\end{equation}
\label{prop:Aprime}\end{proposition}

\begin{proof} The existence of the limit in (\ref{def:Lu}) follows from the convexity of the application $t\mapsto\A(u+t\varphi)$. By convexity also, we have
\begin{equation*}
\A(u+t(v-u))\leq (1-t)\A(u)+t\A(v),\quad 0\leq t\leq 1,
\end{equation*}
whence
\begin{equation*}
\A(u)+\ds\frac{1}{t}\left(\A(u+t(v-u))-\A(u)\right)\leq \A(v),\quad 0< t\leq 1,
\end{equation*}
which gives \refe{eq:AconvexOrdre1} at the limit $t\to 0$.
\end{proof}
We stress out the fact $ \mathcal L(u)$ is not linear, since we might not have 
$$ {\mathcal L} (u)(-\varphi)=-{\mathcal L} (u)(\varphi)$$
for all $\vphi$ (for instance if $\vphi$ is the characteristic function of a set $A$).

With the definition of ${\mathcal L} (u)$ given by Proposition~\ref{prop:Aprime}, it is readily seen that local minimizers of $\J_0\colon u\mapsto\A(u)-\int_\Omega Hu\, dx$ in $\mathrm{BV}(\Omega)$ satisfy 
\begin{equation} 
{\mathcal L} (u)(\vphi) \geq \int_\Omega H\vphi\qquad \mbox{ for all } \vphi\in \mathrm{BV}(\Omega).
\label{eq:EulerLagrangeareaH}\end{equation}
There is equality in \refe{eq:EulerLagrangeareaH} if $u$ and $\vphi$ are smooth enough, but strict inequality if, for instance, $\vphi=\vphi_A$ and
$\J_0(u)<\J_0(u+t\varphi_A)$ for a $t>0$ since 
\begin{equation*}
\ds\frac{1}{t}\left(\A(u+t\varphi_A)-\A(u)\right)=P(A),\forall t>0.
\end{equation*}

We thus consider the following definition:
\begin{definition}[Weak solution]  A function $u\in L^{p+1}\cap \mathrm{BV}(\Omega) $ is said to be a weak solution of \refe{eq:1} if it satisfies
\begin{equation}\label{P1}
\left\{ 
\begin{array}{ll}
\ds {\mathcal L} (u)(\varphi)\geq \int_\Omega [H+\lambda f(u)]\varphi\, dx, \quad\forall\varphi\in L^{p+1}\cap \mathrm{BV}(\Omega)\mbox{ with }\varphi=0\mbox{ on }\partial\Omega, \\[8pt]
u\geq 0 \quad \mbox{ in } \Omega,\qquad   \\[5pt]
u=0 \quad \mbox{ on }\pa\Omega.
\end{array}
\right.
\end{equation}
Furthermore, a weak solution will be said to be {\rm minimal} if it is the smallest among all non-negative weak solutions.
\label{defsol}\end{definition}

\begin{remark}[Local minimizer and weak solution] 
With this definition, it is readily seen that a local minimizer $u$ of $\J_\lambda$ in $L^{p+1}\cap \mathrm{BV}(\Omega)$ satisfying $u=0$ on $\partial\Omega$ and $u\geq 0$ in $\Omega$ is a weak solution of~\refe{eq:1}.
\label{rk:defsol}\end{remark}

Note that the boundary condition in Definition~\ref{defsol} makes sense because functions in $ \mathrm{BV}(\Omega)$ have a unique trace in $L^1(\pa\Omega)$ if  $\pa\Omega$ is Lipschitz (see \cite{GiustiBook}).
\medskip



\subsection{Classical solutions}\label{sec:classicalsol}

A classical solution of (\ref{eq:1}) is a function $u\in \mathcal C^2(\overline \Omega)$ which satisfies equation (\ref{eq:1}) pointwise.

In the case of the semi-linear equation (\ref{eq:laplace}), it is well known that it is enough to show that a  weak solution $u$ is in $L^\infty(\Omega)$, to deduce that it is a classical solution of (\ref{eq:laplace}) (using, for instance, Calderon-Zygmund inequality and a bootstrap argument).

Because of the degenerate nature of the mean curvature operator, an $L^\infty$ bound on $u$ is not enough to show that it is a classical solution of (\ref{eq:1}). 
When $H+\lambda f(u)$ is bounded in $L^\infty$, classical results of the calculus of variation (see \cite{Ma} for instance),  imply that for $n\leq 6$, the surface $(x,u(x))$, the graph of $u$, is $\mathcal C^\infty$ (analytic if $H$ is analytic) and that $u$ is continuous almost everywhere in $\Omega$.
However, to get further regularity on $u$ itself, we need to show that $u$ is Lipschitz continuous on $\Omega$,
as shown by the following proposition.
In the rest of our paper we will thus focus in particular on the Lipschitz regularity of weak solutions. 
\begin{proposition} Assume that $H$ satisfies the conditions of Theorem \ref{mainthm}, and let $u\in L^{p+1}\cap \mathrm{BV}(\Omega)$ be a weak solution of~\refe{eq:1} for some $\lambda>0$. If $u\in\Lip(\Omega)$, then $u$ is a classical solution of (\ref{eq:1}).
In particular, $u\in \mathcal C^{2,\alpha}(\overline{\Omega})$ for all $\alpha\in(0,1)$ and $u$ satisfies $-\div(Tu)=H+\lambda f(u)$ in $\Omega$, $u=0$ on $\pa\Omega$.
\label{prop:cassicalsol}\end{proposition}

\begin{proof}
This result follows from fairly classical arguments of the theory of prescribed mean curvature surfaces and elliptic equations (see for instance \cite{GT}). Anticipating a little bit, we can also notice that (modulo the regularity {\it up to} the boundary) it will be a consequence of Theorem \ref{thm:gia} (ii) below (with $\overline H=H+\lambda f(u)$  instead of $H$), using the characterization of weak solutions given in
Lemma~\ref{lem:min}~(ii).
\end{proof}

\subsection{Main results}\label{sec:mainresults}

Before we state our main results, we recall the following theorem concerning the case  $\lambda=0$, which plays an important role in the sequel:
\begin{theorem}[Giaquinta \cite{Giaquinta74}]\label{thm:gia}
\item[(i)] Let $\Omega$ be a bounded domain with Lipschitz boundary and assume that $H(x)$ is a measurable function such that  (\ref{eq:AA}) holds for some $\eps_0>0$. Then the functional 
$$\F_0(u):=\A(u) - \int_\Omega H(x) u(x)\, dx + \int_{\pa\Omega} |u|\, d\H^{n-1}
$$ 
has a minimizer $u$ in $ \mathrm{BV}(\Omega)$.

\item[(ii)] Furthermore, if $\pa\Omega$ is $\mathcal C^1$,  $H(x)\in\Lip(\Omega)$ and
\begin{equation}\label{eq:B} 
|H(y)|\leq(n-1) \Gamma(y) \quad \mbox{ for all } y\in\pa\Omega
\end{equation}
where $\Gamma(y)$ denotes the mean curvature of $\pa\Omega$ (with respect to the inner normal), then the unique minimizer of $\F_0$ belongs to $\mathcal C^{2,\alpha}(\Omega)\cap\mathcal C^0(\overline \Omega)$ for all $\alpha\in[0,1)$ and is solution to
\begin{equation}\label{eq:10} 
\left\{
\begin{array}{ll}
-\mathrm{div} (Tu) = H &\mbox{ in } \Omega,\\[5pt]
u=0 & \mbox{ on } \pa\Omega ,
\end{array}
\right.
\end{equation}
\item[(iii)]Finally, if $\pa\Omega$ is $ \mathcal C^3$ and the hypotheses of (ii) hold, then $u\in \Lip(\Omega)$.
\end{theorem}

\vspace{10pt}

The key in the proof of {\it (i)} is the fact that (\ref{eq:AA}) and the coarea formula for $\mathrm{BV}$ functions yield
$$ \eps_0\int_\Omega |Du| \leq \int_\Omega |Du| - \int_\Omega H(x) u(x)\, dx $$
for all $u\in  \mathrm{BV}(\Omega)$. 
This is enough to guarantee the existence of a minimizer.
The condition (\ref{eq:B}) is a sufficient condition for the minimizer to satisfy $u=0$ on $\pa\Omega$.
In the sequel, we assume that $\Omega$ is such that (\ref{eq:AA}) holds, as well as the following strong version of (\ref{eq:B}):
\begin{equation}\label{eq:BB} 
|H(y)|\leq(1-\eps_0)(n-1) \Gamma(y) \quad \mbox{ for all } y\in\pa\Omega.
\end{equation}

\vspace{10pt}

\begin{remark}\label{rmk:1}
When  $H(x)=H_0$ is constant, Serrin proves in \cite{serrin} that (\ref{eq:B}) is necessary for the equation $ -\mathrm{div} (Tu) = H$ to have a solution for {\rm any} smooth boundary data. However, it is easy to see that (\ref{eq:B}) is not always necessary for (\ref{eq:10}) to have a solution: when $\Omega=B_R$ and $H=\frac{n}{R}$, (\ref{eq:10}) has an obvious solution given by an upper half sphere, even though  (\ref{eq:B}) does not hold since $(n-1)\Lambda =(n-1)/R < H=n/R$.

Several results in this paper only require Equation (\ref{eq:10}) to have a solution with $(1+\delta)H$ in the right-hand side instead of $H$. In particular, this is enough to guarantee the existence of a minimal branch of solutions and the existence of an extremal solution.
When $\Omega=B_R$,  we can thus replace (\ref{eq:BB}) with
$$
|H(y)|\leq(1-\eps_0)n \Gamma(y) \quad \mbox{ for all } y\in\pa B_R.
$$
However, the regularity theory for the extremal solution will require the stronger assumption (\ref{eq:BB}).
\end{remark}

Finally, we assume that there exists a constant $H_0>0$ such that:
\begin{equation}\label{eq:CC}  
H\in \Lip(\Omega) \mbox{ and } H(x)\geq H_0>0 \mbox{ for all $x\in\Omega$}.
\end{equation}
This last condition will be crucial in the proof of Lemma \ref{lem:ex} to prove the existence of a non-negative solution for small values of $\lambda$. Note that  Pucci and Serrin \cite{ps86} proved, using a generalization of Pohozaev's Identity, that if $H=0$ and $p\geq (n+2)/(n-2)$, then (\ref{eq:1}) has no non-trivial solutions for any values of $\lambda>0$ when $\Omega$ is star-shaped (see also  Bidaut-V\'eron \cite{MFBV}).

\vspace{10pt}

Our main theorem is the following:
\begin{theorem} 
Let $\Omega$ be a bounded subset of $\RR^n$ such that $\pa \Omega$ is $\mathcal C^3$.
Assume that $H(x)$ satisfies conditions \refe{eq:AA}, \refe{eq:BB} and \refe{eq:CC}. Then, there exists $\lambda^*>0$ such that:
\begin{itemize}
\item[(i)] For all $\lambda \in [0,\lambda^*]$,  (\ref{eq:1}) has one minimal weak solution $u_\lambda$.
\item[(ii)] For $\lambda > \lambda^*$,   (\ref{eq:1}) has no weak solution.
\item[(iii)] The application $\lambda \mapsto u_\lambda$ is non-decreasing.
\end{itemize}
\label{mainthm}\end{theorem}

The proof of Theorem \ref{mainthm} is done in two steps: First we  show that the set of $\lambda$ for which a weak solution exists is a non-empty bounded interval (see Section~\ref{sec:existence1}).  Then we prove the existence of the extremal solution for $\lambda=\lambda^*$ (see Section~\ref{sec:existence2}). The key result in this second step is the following uniform $L^\infty$ estimate:
\begin{proposition}\label{prop:infty0}
There exists a constant $C$ depending only on $\Omega$ and $H$, such that the minimal weak solution $u_\lambda$ of \refe{eq:1} satisfies 
$$ \|u_\lambda\|_{L^\infty(\Omega)} \leq C \qquad \mbox{ for all } \lambda\in[0,\lambda^*] .$$
\end{proposition}

\medskip

Next we investigate the regularity of  minimal weak solutions:
We want to show that minimal weak solutions are classical solutions of \refe{eq:1} (in view of Proposition \ref{prop:cassicalsol}, we need to obtain a Lipschitz estimate).
This, it seems, is a much more challenging problem and we obtain some  results only in the radially symmetric  case.  
More precisely, we  show the following:
\begin{theorem}\label{thm:rad}
Assume that $\Omega=B_R\subset \RR^n$ ($n\geq 1$), $H=H(r)$, and that the conditions of Theorem \ref{mainthm} hold.
Then the minimal weak solution of \refe{eq:1} is radially symmetric and lies in $\Lip(\Omega)$. In particular there exists a constant $C$ such that
\begin{equation}\label{eq:gradrad} 
|\na u_\lambda(x)|\leq \frac{C}{\lambda^*-\lambda} \quad \mbox{ a.e. in }\Omega,\quad \forall \lambda\in[0,\lambda^*).
\end{equation}
In particular $u_\lambda$ is a classical solution of (\ref{eq:1}), and if $H(x)$ is analytic in $\Omega$, then $u_\lambda$ is analytic in $\Omega$ for all $\lambda<\lambda^*$.
\end{theorem}

Note that the Lipschitz constant in  \refe{eq:gradrad} blows up as $\lambda\to\lambda^*$.
However, we are then able to show the following:
\begin{theorem}\label{thm:d2}
Assume that the conditions of Theorem \ref{thm:rad} hold. Then
there exists a constant $C$ such that for any $\lambda\in[0,\lambda^*]$, the minimal weak solution $u_\lambda\in \Lip(\Omega)$ and satisfies
\begin{equation*}
| \na u_\lambda(x)|\leq C\quad \mbox{ a.e. in  }  \Omega.
\end{equation*}
In particular the extremal solution $u^*$ is a classical solution of (\ref{eq:1}).
\end{theorem} 

The classical tools of continuation theory developed for example in \cite{KK,CR} can be modified in our context (non-linear leading order differential operator, radial case) to show that there exists a second branch of solution in the neighborhood of $\lambda^*$: 
\begin{theorem}\label{thm:CR}
Assume that the conditions of Theorem \ref{thm:rad} hold. Then there exists $\delta>0$ such that
for $\lambda^*-\delta<\lambda<\lambda^*$ there are at least two classical solutions to~\refe{eq:1}.
\end{theorem}

To prove this result, we will need to consider the linearized operator
\begin{equation*}
L_\lambda(v)= -\partial_i(a^{ij}(\nabla u_\lambda)\partial_j v)-\lambda f'(u_\lambda)v
\end{equation*}
where
\begin{equation*}
a^{ij}(\p)=\frac{1}{(1+|\p|^2)^{1/2}}\left(\delta_{ij}-\frac{\p_i\p_j}{1+|\p|^2}\right),\qquad \p\in\R^n.
\end{equation*}
If we denote by $\mu_1(\lambda)$ the first eigenvalue of $L_\lambda$, we will prove in particular:
\begin{lemma}
Assume that the conditions of Theorem \ref{thm:rad} hold. 
Then the linearized operator $L_\lambda$ has positive first eigenvalue $\mu_1(\lambda)>0$ for all $\lambda \in [0,\lambda^*)$.
Furthermore, the linearized operator $L_{\lambda^*}$ corresponding to the extremal solution has zero first eigenvalue $\mu_1(\lambda^*)=0$, and $\lambda^*$ corresponds to a turning point for the $(\lambda,u_\lambda)$ diagram.
\end{lemma}
A turning point means that there exists a parametrized family of classical solutions 
\begin{equation*}
s\mapsto (\lambda(s),u(s)), \qquad s\in(-\eps,\eps)
\end{equation*}
with $\lambda(0)=\lambda^*$ and $\lambda(s)<\lambda^*$ both for $s<0$ and $s>0$. 
In particular we will prove that  $\lambda'(0)=0$ and $\lambda''(0)<0$.

In the radially symmetric case, we can thus summarize our results in the following corollary: 
\begin{corollary}
Assume that $\Omega=B_R\subset \RR^n$ ($n\geq 1$), $H=H(r)$, and that the conditions of Theorem \ref{mainthm} hold. 
Then there exists $\lambda^*>0$, $\delta>0$ such that
\begin{enumerate}
\item if $\lambda>\lambda^*$, there is no weak solution of (\ref{eq:1}),
\item if $\lambda\leq \lambda^*$, there is a minimal classical solution of ~\refe{eq:1}.
\item if $\lambda^*-\delta<\lambda<\lambda^*$, there are at least two classical solutions  of~\refe{eq:1}.
\end{enumerate}
\end{corollary} 
Finally, we point out that numerical computation suggest that for some values of $n$ and $H$, a third  branch of solutions may arise (and possibly more). 
\medskip

The paper is organized as follows: In Section~\ref{sec:weaksol}, we give some a priori  properties of weak solutions. In Section~\ref{sec:existence1} we show the existence of a branch of minimal weak solutions for $\lambda\in[0,\lambda^*)$. 
We then establish, in Section~\ref{sec:linfty}, a uniform $L^\infty$ bound for these minimal weak solutions (Proposition \ref{prop:infty0}), which we use, in Section~\ref{sec:existence2}, to show the existence of an extremal solution as $\lambda\to\lambda^*$ (thus completing the proof of Theorem \ref{mainthm}).  
In the last Section~\ref{sec:reg} we prove the regularity of the minimal weak solutions, including that of $u_{\lambda^*}$, in the radial case (Theorems \ref{thm:rad} and \ref{thm:d2}) and we give the proof of Theorem~\ref{thm:CR}. 
In appendix, we prove a comparison lemma that is used several times in the paper.

\begin{remark} One might want to generalize those results  to other non-linearity $f(u)$:
In fact, all the results presented here holds (with the same proofs) if $f$ is a $\mathcal C^2$ function satisfying:
\begin{itemize} 
\item [(H1)] $f(0)=0$, $f'(u)\geq 0$ for all $u\geq 0$.
\item [(H2)] There exists $C$ and $\alpha >0$ such that $f'(u) \geq \alpha $ for all $u\geq C$.
\item [(H3)] If $u\in L^q(\Omega)$ for all $q\in [0,\infty)$ then $f(u)\in L^n(\Omega)$.
\end{itemize}
The last condition, which is used to prove the $L^\infty$ bound (and the Lipschitz regularity near $r=0$) of the extremal solution $u_{\lambda^*}$ is the most restrictive. It excludes in particular non-linearities of the form $f(u)=e^u-1$.
However, similar results hold also for such non-linearities, though  the proof of Proposition~\ref{prop:infty0} has to be modified in that case. This will be developed in a forthcoming paper.
\item We can  also consider right-hand sides of the form $\lambda (1+u)^p$ (or $\lambda e^u$). In that case, Theorem~\ref{mainthm}, Proposition~\ref{prop:infty0} and Theorem~\ref{thm:rad} are still valid (but require different proofs), but  Theorem~\ref{thm:d2} is not. Indeed, our proof of the boundary regularity of the extremal solution $u_{\lambda^*}$ (Lemma \ref{lem:bd2}) relies heavily on condition~\refe{eq:BB}, which should be replaced here by the condition
\begin{equation}
\lambda^*<(n-1)\Gamma(y) \qquad \mbox{ for all }y\in\partial\Omega.
\label{boundarylambda}\end{equation}
However, it is not clear that $\lambda^*$ should satisfy \refe{boundarylambda}.
\end{remark}


\section{Properties of weak solutions}\label{sec:weaksol}


\subsection{Weak solutions as global minimizers}


Non-negative minimizers of $\J_\lambda$ that satisfy $u=0$ on $\pa\Omega$ are in particular critical points of $\J_\lambda$, and thus weak solutions  of (\ref{eq:1}). But not all critical points are minimizers. However, the convexity of the perimeter yields the following result:
\begin{lemma} \label{lem:min} Assume that $\pa\Omega$ is $\mathcal{C}^1$ and let $u$ be a non-negative function in $L^{p+1}\cap \mathrm{BV}(\Omega)$.
The following propositions are equivalent:
\item[(i)] $u$ is  a weak solution of (\ref{eq:1}),
\item[(ii)] $u=0$ on $\pa\Omega$ and for every $v\in L^{p+1}\cap \mathrm{BV}(\Omega)$, we have 
\begin{equation}\label{eq:almostVariii0}
\A(u)-\int_\Omega (H+\lambda f(u))\,u \,dx\leq\A(v)-\int_\Omega (H+\lambda f(u))\,v\, dx+\int_{\pa\Omega} |v|\,d\H^{N-1},
\end{equation}
\item[(iii)] $u=0$ on $\pa\Omega$ and for every $v\in L^{p+1}\cap \mathrm{BV}(\Omega)$, we have
\begin{equation}
  \J_\lambda(u)  \leq  \J_\lambda(v)  +\int_\Omega \lambda G(u,v)\, dx+\int_{\pa\Omega} |v|\,d\H^{N-1}
\label{eq:almostVariii}\end{equation}
where 
$$G(u,v) = F(v)-F(u)-f(u)(v-u)\geq 0.$$
\end{lemma}
In particular, {\it (ii)} implies that any weak solution $u$ of (\ref{eq:1}) is a global minimizer in $L^{p+1}\cap \mathrm{BV}(\Omega)$ of the functional (which depends on $u$)
$$
\F^{[u]}_\lambda(v):=\A(v)-\int_\Omega (H+\lambda f(u))\,v\, dx+\int_{\pa\Omega} |v|\,d\H^{N-1}.
$$
Furthermore, since $G(u,u)=0$, {\it (iii)} implies  that any weak solution $u$ of (\ref{eq:1}) is also a global minimizer in $L^{p+1}\cap \mathrm{BV}(\Omega)$ of the functional
\begin{eqnarray*}
\G^{[u]}_\lambda(v)& := & \A(v)-\int_\Omega Hv+\lambda F(v)\, dx  +\int_{\pa \Omega} |v|\, d\H^{N-1}+\int_\Omega\lambda G(u,v)\, dx\\
& =& \J_\lambda(v)+\int_{\pa \Omega} |v|\, d\H^{N-1}+\int_\Omega\lambda G(u,v)\, dx.
\end{eqnarray*}
\medskip

\begin{proof}[Proof of Lemma \ref{lem:min}]
The last two statements {\it (ii)} and {\it (iii)} are clearly equivalent (this follows from a simple computation using the definition of $G$). 

Next, we notice that if {\it (ii)} holds, then taking $v=u+t\varphi$ in \refe{eq:almostVariii0}, where $\varphi\in L^{p+1}\cap \mathrm{BV}(\Omega)$ with $\varphi=0$ on $\partial\Omega$, we get
$$
\frac{1}{t}(\A(u+t\vphi)-\A(u)) \geq \int (H+\lambda f(u))\vphi\, dx.
$$
Passing to the limit $t\to 0$, we deduce ${\mathcal L} (u)(\varphi)\geq\int_\Omega (H+f(u))\varphi dx$, i.e. $u$ is a solution of (\ref{P1}). 
In view of Definition \ref{defsol}, we thus have {\it (ii)}$\Rightarrow${\it (i)}.

So it only remains to prove that  {\it (i)} implies {\it (ii)}, that is  
$$\F^{[u]}_\lambda(u)=\min_{v\in L^{p+1}\cap \mathrm{BV}(\Omega)}\F^{[u]}_\lambda(v).$$ 
By definition of weak solutions, we have
$${\mathcal L} (u)(\varphi)\geq\int_\Omega (H+\lambda f(u))\varphi\, dx$$ 
for all $\varphi\in L^{p+1}\cap \mathrm{BV}(\Omega)$ with $\varphi=0$ on $\partial\Omega$. 
Furthermore, by \refe{eq:AconvexOrdre1}, we have 
\begin{equation*}
\A(u)+{\mathcal L} (u)(v-u) \leq\A(v),
\end{equation*}
for every $v\in L^{p+1}\cap \mathrm{BV}(\Omega)$ with $v=0$ on $\partial\Omega$. 
We deduce (taking $\vphi=v-u$):
\begin{equation*}
\A(u)+\int_\Omega (H+\lambda f(u))(v-u)\, dx \leq\A(v),
\end{equation*}
which implies
\begin{equation}
\F^{[u]}_\lambda(u)\leq\F^{[u]}_\lambda(v)
\label{mintrace0}\end{equation} 
for all $v\in L^{p+1}\cap \mathrm{BV}(\Omega)$ satisfying  $v=0$ on $\partial\Omega$. 

It thus only remains to show that  \refe{mintrace0} holds even when $v\neq 0$ on $\partial\Omega$.
For that, the idea is to apply \refe{mintrace0} to the function  $v-w^\eps$ where $(w^\eps)$ is a sequence of functions in $L^{p+1}\cap \mathrm{BV}(\Omega)$ converging to $0$ in $L^{p+1}(\Omega)$ such that $w^\eps=v$ on $\partial\Omega$. 
Heuristically the mass of $w^\eps$ concentrates on the boundary $\partial\Omega$ as $\eps$ goes to zero, and so  $\A(v-w^\eps)$ converges to
$\A(v)+\int_{\partial\Omega}|v|d\H^{N-1}$. 
This type of argument is fairly classical, but we give a detailed proof below, in particular to show how one can pass to the limit in the non-linear term.

\medskip

First, we consider  $v\in L^\infty\cap\mathrm{BV}(\Omega)$. Then, for every  $\eps>0$, there exists $w^\eps\in L^\infty\cap\mathrm{BV}(\Omega)$ such that 
$
w^\eps=v\mbox{ on }\pa\Omega$ satisfying the estimates:
\begin{equation}
 \|w^\eps\|_{L^1(\Omega)}\leq \eps\ds\int_{\pa\Omega}|v|d\H^{N-1},\quad \ds\int_\Omega|Dw^\eps|\leq (1+\eps) \ds\int_{\pa\Omega}|v|d\H^{N-1}
\label{eq:wboundary}\end{equation}
and $\|w^\eps\|_{L^\infty(\Omega)}\leq 2\|v\|_{L^\infty(\Omega)}$ (see Theorem~2.16 in \cite{GiustiBook}). In particular we note that
\begin{equation}
\|w^\eps\|^{p+1}_{L^{p+1}(\Omega)}\leq 2^p\|v\|_{L^\infty(\Omega)}^p\|w^\eps\|_{L^1(\Omega)}\to 0
\label{limw}\end{equation}
when $\eps\to 0$. Using \refe{mintrace0}, \refe{eq:wboundary} and the fact that $\A(v-w^\eps)\leq\A(v)+\int_\Omega|Dw^\eps|$, we deduce:
\begin{align}
\F^{[u]}_\lambda(u)&\leq\F^{[u]}_\lambda(v-w^\eps)\nonumber\\
&\leq \A(v)-\int_\Omega (H+\lambda f(u))\,v\, dx +\ds\int_\Omega|Dw^\eps|+\int_\Omega(H+ f(u))w^\eps\, dx\nonumber\\
&\leq \A(v)-\int_\Omega (H+\lambda f(u))\,v\, dx+(1+\eps)\ds\int_{\pa\Omega}|v|d\H^{N-1}\nonumber \\
&\qquad +\|w^\eps\|_{L^{p+1}}\|H+f(u)\|_{L^{\frac{p+1}{p}}} \nonumber \\ 
&=\F^{[u]}_\lambda(v)+\eps\ds\int_{\pa\Omega}|v|d\H^{N-1}+\|w^\eps\|_{L^{p+1}}\|H+f(u)\|_{L^{\frac{p+1}{p}}}.\label{endiVSiii}
\end{align}
(Note that $f(u)\in L^{\frac{p+1}{p}}(\Omega)$ since $u\in L^{p+1}(\Omega)$).
Using \refe{limw} and taking the limit $\eps\to 0$ in \refe{endiVSiii},  we obtain \refe{mintrace0} for any $v\in L^\infty\cap \mathrm{BV}(\Omega)$.
\smallskip

We now take $v\in L^{p+1}\cap \mathrm{BV}(\Omega)$.  Then, the computation above shows that for every $M>0$ we have:
\begin{equation*}
\F^{[u]}_\lambda(u)\leq\F^{[u]}_\lambda(T_M(v)), 
\end{equation*}
where $T_M$ is the truncation  operator $T_M(s):=\min(M,\max(s,-M))$.
Clearly, we have
$T_M(v)\to v$ in $L^{p+1}(\Omega)$ as $M\to \infty$. 
Furthermore, one can show that $\A(T_M(v))\to\A(v)$. 
As a matter of fact, the lower semi-continuity of the perimeter gives $\A(v)\leq\ds\liminf_{M\to+\infty}\A(T_M(v))$, and the coarea formula implies:
\begin{align*}
\A(T_M(v))&\leq\A(v)+\ds\int_\Omega |D(v-T_M(v))|\\
&=\A(v)+\ds\int_0^{+\infty} P(\{v-T_M(v)>t\})dt \\
&=\A(v)+\ds\int_M^{+\infty} P(\{v>t\}) dt\\
&\longrightarrow \A(v)\mbox{ when }M\to+\infty.
\end{align*}
We deduce that $\F^{[u]}_\lambda(T_M(v))\longrightarrow\F^{[u]}_\lambda(v)$, and the proof is complete. 
\end{proof}


\subsection{A priori bounds}


Next, we want to derive some a priori bounds satisfied by any weak solution $u$ of \refe{eq:1}.

First, we have the following lemma:
\begin{lemma}\label{lem:A}
Let $u$ be a weak solution of (\ref{eq:1}), then
\begin{equation*}
 \int_A H+\lambda f(u)\, dx \leq P(A)
 \end{equation*}
for all measurable sets $A\subset\Omega$.
\end{lemma}
\begin{proof}
When $u$ is smooth, this lemma can be proved by integrating (\ref{eq:1}) over the set $A$ and noticing that $|\frac{\na u\cdot \nu}{\sqrt{1+|\na u|^2}}| \leq 1$ on $\pa A$.
If $u$ is not smooth, we  use Lemma \ref{lem:min} (ii): for all $A\subset\Omega$, we get (with $v=\vphi_A$):
\begin{equation*}
\A(u)-\int_\Omega [H+\lambda f(u)] u \leq \A(u+\vphi_A) -\int_\Omega [H+\lambda f(u)] (u+\vphi_A)  + \H^{n-1}(\pa\Omega\cap A) .
\end{equation*} 
We deduce
\begin{equation*}
0\leq \int_\Omega |D\vphi_A| +  \H^{n-1}(\pa\Omega\cap A) -\int_A H+\lambda f(u) \, dx.
\end{equation*}
and so
\begin{equation*}
0\leq P( A)-\int_A H+\lambda f(u)\, dx. \qedhere
\end{equation*}
\end{proof}

Lemma \ref{lem:A} suggests that $\lambda$ can not be too large for  (\ref{eq:1}) to have a weak solution. In fact, it provides an upper bound on $\lambda$, if we know that $\int_\Omega u\, dx$ is bounded from below. This is proved in the next lemma: 
\begin{lemma}[Bound from below] Let $u$ be a  weak solution of  (\ref{eq:1}) for some $\lambda \geq 0$.  Then 
\begin{equation*}
u\geq \underline{u} \quad \mbox{ in } \Omega
\end{equation*}
where $\underline{u}$ is the solution corresponding to $\lambda=0$:
\begin{equation}\label{eq:uo}\tag{\mbox{$P_{0}$}}
\left\{
\begin{array}{r c l l}
-\div(T\underline{u})&=&H & \mbox{ in } \Omega,\\
\underline{u}&=&0 & \mbox{ on }\pa\Omega.
\end{array}
\right.
\end{equation}
\label{lem:boundbelow}\end{lemma}

\begin{proof}
For $\delta\geq 0$, let $u_{\delta}$ be the solution to the problem
\begin{equation}\label{eq:ueps}\tag{\mbox{$P_{\delta}$}}
\left\{
\begin{array}{r c l l}
-\div(Tu)&=&(1-\delta)H & \mbox{ in } \Omega,\\
u&=&0 & \mbox{ on }\pa\Omega.
\end{array}
\right.
\end{equation}
Problem~\refe{eq:ueps} has a solution $u_{\delta}\in \Lip(\Omega)$ (by Theorem \ref{thm:gia}) and  $(u_{\delta})$ is increasing to $\underline{ u}$ when $\delta\downarrow 0$. 
We also recall~\cite{Giusti76} that the function $u_{\delta}$ is the unique minimizer in $L^{p+1}\cap \mathrm{BV}(\Omega)$ of the functional 
$$\F_{\delta}(u) = \int_\Omega\sqrt{1+|\na u|^2 }-\int_\Omega (1-\delta) H(x)u(x)\, dx +\int_{\pa\Omega} |u|.$$
The lemma then follows easily from the comparison principle,  Lemma~\ref{lem:comparison}: Taking $G_-(x,s)=-(1-\delta)H(x)s$, $G_+(x,s)=-H(x)s-\lambda F(s)+\lambda G(u(x),s)$, $K_-=K_+=L^{p+1}\cap \mathrm{BV}(\Omega)$, Lemma~\ref{lem:comparison} implies:
\begin{eqnarray*}
0&\leq& \int_\Omega -\delta H (\max(u_\delta,u)-u)+ \lambda [F(u)-F( \max(u,u_{\delta}))+G(u,\max(u,u_{\delta}))]\\
&= & -\int_\Omega \left(\delta H+ \lambda f(u)\right) (u_{\delta}-u)_+,
\end{eqnarray*}
where $v_+=\max(v,0)$. Since $H>0$ and $u\geq 0$ in $\Omega$, this implies $u_{\delta}\leq u$ a.e. in $\Omega$. Taking the limit $\delta\to 0$, we obtain $\underline{u}\leq u$ a.e. in $\Omega$.
\end{proof}

As a corollary to Lemma~\ref{lem:A} and Lemma~\ref{lem:boundbelow}, we have the following a priori bound on $\lambda$:

\begin{lemma}[A priori bound]
If (\ref{eq:1}) has a  weak solution for some $\lambda \geq 0$, then
\begin{equation*}
\lambda \leq \frac{P(\Omega)-\int_\Omega H\, dx}{\int_\Omega \underline{u}\, dx}
\end{equation*}
with $\underline{u}$ solution of \refe{eq:uo}.
\label{lem:lambdabounded}\end{lemma}


\section{Existence of minimal weak solutions for $\lambda\in[0,\lambda^*)$}\label{sec:existence1}


In this section, we begin the proof of Theorem \ref{mainthm} by showing the following proposition: 
\begin{proposition}\label{prop:2}
Let $\Omega$ be a bounded subset of $\RR^n$ such that $\pa \Omega$ is $\mathcal C^3$.
Assume that $H(x)$ satisfies conditions \refe{eq:AA}, \refe{eq:BB} and \refe{eq:CC}. Then, there exists $\lambda^*>0$ such that:
\begin{itemize}
\item[(i)] For all $\lambda \in [0,\lambda^*)$,  (\ref{eq:1}) has one minimal weak solution $u_\lambda$.
\item[(ii)] For $\lambda > \lambda^*$,   (\ref{eq:1}) has no weak solution.
\item[(iii)] The application $\lambda \mapsto u_\lambda$ is non-decreasing.
\end{itemize}
\end{proposition}
To establish Theorem \ref{mainthm}, it will thus only remain to show the existence of an extremal solution for $\lambda=\lambda^*$. This will be done in Section \ref{sec:existence2}.
To prove Proposition \ref{prop:2}, we will first show that weak solutions exist for small values of $\lambda$. 
Then, we will prove that the set of the values of $\lambda$ for which weak solutions exist is an interval.

\subsection{Existence of weak solutions for small values of $\lambda$}

We start with the following lemma:
\begin{lemma}\label{lem:ex}
Suppose that \refe{eq:AA}, \refe{eq:BB} and \refe{eq:CC} hold. Then there exists  $\lambda_0>0$ such that (\ref{eq:1}) has a weak solution for all $\lambda<\lambda_0$.
\end{lemma}

Note that Lemma~\ref{lem:ex} is proved by Huisken in \cite{Huisken84} (see also \cite{stone}) in the case $p=1$. Our proof is slightly different from those two references and relies on the fact that $H>0$.

\begin{proof}
We will show that for small $\lambda$, the functional $\J_{\lambda}$ has a local minimizer in $L^{p+1}\cap\mathrm{BV}(\Omega)$ that satisfies $u=0$ on $\pa\Omega$.
Such a minimizer is a critical  point for $\J_\lambda$, and thus (see Remark~\ref{rk:defsol}) a weak solution of (\ref{eq:1}).
\smallskip

Let $\delta$ be a small parameter such that $(1+\delta)(1-\eps_0)<1$ where $\eps_0$ is defined by the conditions  (\ref{eq:AA}) and  (\ref{eq:BB}).
Then there exists $\eps'>0$ such that
$$ 
\left|\int_A (1+\delta) H\, dx\right| \leq (1+\delta)(1-\eps_0)\H^{n-1}(\pa A)\leq (1-\eps')P(A),
$$
and
$$ |(1+\delta) H(y)| \leq (1-\eps')(n-1)\Gamma(y) \qquad \forall y\in\pa\Omega.$$
Theorem \ref{thm:gia}  thus gives the existence of $w\geq 0$ local minimizer in $\mathrm{BV}(\Omega)$ of 
$$\G_\delta (u) =\A(u)- \int_\Omega (1+\delta) H(x) u \, dx + \int _{\pa \Omega} |u|\, d\sigma(x),$$
with $w\in \mathcal C^{2,\alpha}(\overline \Omega)$ and $w=0$ on $\pa\Omega$.

It is readily seen that  the functional  $\J_\lambda$ has a global minimizer $u$ in
$$ K=\{v\in L^{p+1}\cap \mathrm{BV}(\Omega)\, ;\, 0\leq v\leq w +1 \}.$$
We are now going to show that if $\lambda$ is small enough, then $u$ satisfies
\begin{equation}\label{eq:uw}
u(x)\leq w (x)\quad \mbox{ in }\Omega.
\end{equation}
For this,  we use the comparison principle (Lemma~\ref{lem:comparison}) 
with $G_-(x,s)=-H(x)s-\lambda F(s)$ and $G_+(x,s)=-(1+\delta)H(x)s$ (i.e. $\F_-=\J_\lambda$ and $\F_+=\G_\delta$), and $K_-=L^{p+1}\cap \mathrm{BV}(\Omega)$, $K_+=K$. Since $\max(u, w)\in K$, we obtain
\begin{align*}
0&\leq\int_\Omega -\delta H (\max(u, w)-w)+\lambda (F(\max(u, w))-F(w))\, dx \\
& \leq \int_\Omega -\delta H (\max(u, w)-w)+\lambda \sup_{s\in[0,\|w\|_\infty+1]} |f(s)| (\max(u, w)-w) \, dx\\
&\leq  \int_\Omega -(u-w)_+\left[\delta H-\lambda f(\|w\|_\infty+1) \right]\, dx .
\end{align*}
Therefore, if we take $\lambda$ small enough such that  $\lambda < \delta \frac{\inf H}{f(\|w\|_\infty+1)} = \delta \frac{H_0}{f(\|w\|_\infty+1)}$, we deduce (\ref{eq:uw}).

Finally, (\ref{eq:uw})  implies that  $u=0$ on $\pa\Omega$ and that $u$ is a critical point of $\J_\lambda$  in $L^{p+1}\cap\mathrm{BV}(\Omega)$, which completes the proof.
\end{proof}

\subsection{Existence of $u_\lambda$ for $\lambda<\lambda^*$}

We now define
\begin{equation*}
\lambda^* = \sup\{\lambda \, ;\, (\ref{eq:1}) \mbox{ has a weak solution} \}.
\end{equation*}
Lemmas \ref{lem:lambdabounded} and \ref{lem:ex} imply
\begin{equation*}
0< \lambda^*<\infty.
\end{equation*}
In order to complete the proof of Proposition \ref{prop:2}, we need to show:
\begin{proposition}\label{prop:min}
For all $\lambda\in[0,\lambda^*)$ there exists a minimal weak solution $u_\lambda$ of (\ref{eq:1}).
Furthermore,  the application $\lambda\mapsto u_\lambda$ is non-decreasing.
\end{proposition}
\begin{proof}[Proof of Proposition \ref{prop:min}]
Let us fix $\lambda_1\in[0,\lambda^*)$. 
By definition of $\lambda^*$, there exists $\overline\lambda\in(\lambda_1,\lambda^*]$ such that (\ref{eq:1}) has a weak solution $\overline u \in L^{p+1}\cap\mathrm{BV}(\Omega)$ for $\lambda=\overline \lambda$.

We also recall that $\underline{u}$ denotes the solution to \refe{eq:uo}. We then define the sequence $u_n$ as follows: We take
$$u_0=\underline{u}$$
and for any $n\geq 1$,  we set
\begin{equation*} 
I_n(v)=\A(v) -\int_\Omega [H  +\lambda_1 f(u_{n-1})] v \, dx +\int_{\partial\Omega} |v|
\end{equation*}
and let $u_n$ be the unique minimizer of $I_n$ in $\mathrm{BV}(\Omega)$.
In order to prove Proposition~\ref{prop:min}, we will show that this sequence $(u_n)$ is well defined (i.e. that $u_n$ exists for all $n$), and that  it converges to a weak solution of  ($P_{\lambda_1}$). This will be a consequence of the following Lemma:

\begin{lemma}\label{lem:n}
For all $n\geq 1$, the functional $I_n$ admits a global minimizer $u_{n}$ on $ \mathrm{BV}(\Omega)$. 
Moreover, $u_n\in \Lip(\Omega)$ satisfies
\begin{equation}\label{eq:un} 
\underline{u}\leq u_{n-1}< u_n \leq \overline u\mbox{ in }\Omega.
\end{equation}
\end{lemma}

We can now complete the proof of  Proposition \ref{prop:min}: by Lebesgue's monotone convergence theorem, we get that $(u_n)$ converges almost everywhere and in $L^{p+1}(\Omega)$ to a function $u_\infty$ satisfying
\begin{equation*}
 0\leq u_\infty\leq \overline u.
 \end{equation*}
In particular, we have $u_\infty=0$ on $\pa\Omega$.	
Furthermore, for every $n\geq 0$, we have
\begin{equation*}
 I_n(u_n)\leq I_n(0)=|\Omega|
 \end{equation*}
and so by \refe{eq:AreaVariation},
\begin{equation*}
\int_\Omega |D u_n| \leq 2|\Omega| +\sup(H)\|\overline u\|_{L^1}+\lambda_1\|\overline u\|^{p+1}_{L^{p+1}(\Omega)},
 \end{equation*}
hence, by lower semi-continuity of the total variation, $u_\infty\in L^{p+1}\cap\mathrm{BV}(\Omega)$.
Finally, for all $v\in L^{p+1}\cap\mathrm{BV}(\Omega)$ and for all $n\geq 1$, we have
\begin{equation*}
I_n (u_n)\leq I_n(v)
\end{equation*}
and using the lower semi-continuity of the perimeter, and the strong $L^{p+1}$ convergence, we deduce
\begin{equation*}
\A(u_\infty) -\int Hu_\infty +\lambda_1 f(u_\infty)u_\infty  \,dx
 \leq\A(v)-\int Hv +\lambda_0 f(u_\infty) v\, dx
\end{equation*}
We conclude, using Lemma \ref{lem:min} {\it (ii)}, that $u_\infty$ is a solution of ($P_{\lambda_1}$).
\end{proof}

\bigskip

The rest of this section is devoted to the proof of Lemma \ref{lem:n}:
\smallskip

\begin{proof}[Proof of Lemma \ref{lem:n}]
We recall that $\underline{u}$ denotes the unique minimizer of $\F_0$ in $\mathrm{BV}(\Omega)$ and that, by Lemma \ref{lem:boundbelow}, we have the inequality $\underline{u}\leq \overline u$ a.e. on $\Omega$.

Assume now that we constructed $u_{n-1}$ satisfying $u_{n-1}\in \Lip(\Omega)$ and
\begin{equation*}
\underline{u}\leq u_{n-1}\leq \overline u.
\end{equation*}
We are going to show that $u_n$ exists and satisfies (\ref{eq:un}) (this implies Lemma \ref{lem:n} by first applying the result to $n=1$ and proceeding from there by induction).
\vspace{10pt}

First of all,  Lemma \ref{lem:A} implies
\begin{equation*}
\int_A H+\overline \lambda \, f(\overline u)\, dx \leq P( A)
\end{equation*}
for all measurable sets $A\subset \Omega$.
Since $u_{n-1}\leq \overline u$ and $\lambda_1<\overline \lambda$, we deduce that 
\begin{equation}
 \int_A H+ \lambda_1  f(u_{n-1})\, dx < P(A)
\label{Scondition}\end{equation}
for all measurable sets $A\subset \Omega$.
Following Giusti \cite{Giusti78}, we can then prove (a proof of this lemma is given at the end of this section):
\begin{lemma}\label{lem:eps}
There exists $\eps>0$ such that
\begin{equation*}
\int_A H+ \lambda_1  f(u_{n-1})\, dx < (1-\eps)P( A)
\end{equation*}
for all measurable sets $A\subset \Omega$.
In particular (\ref{eq:AA}) holds with $\overline H=H+\lambda_1 f(u_{n-1})$ instead of $H$
\end{lemma}
This lemma easily implies the existence of a minimizer $u_{n}$ of $I_n$ in $\mathrm{BV}(\Omega)$ (using  Theorem \ref{thm:gia} with  $\overline H$ instead of $H$).
Furthermore, since $u_{n-1}\in \Lip(\Omega)$ and $u_{n-1}=0$ on $\pa\Omega$ condition (\ref{eq:B}) is satisfied with $\overline H$ instead of $H$ and so (by Theorem \ref{thm:gia}):
\begin{equation*}
u_n=0\mbox{ on } \pa\Omega
\end{equation*}
and 
\begin{equation*}
u_n\in \Lip(\Omega).
\end{equation*}

Finally,  we check that the minimizer $u_n$ satisfies 
$$\underline{u}\leq u_n\leq\overline{u}.$$ 
Indeed, the first inequality is a consequence of the comparison Lemma~\ref{lem:comparison} applied to $\F_-=\F_0$, $\F_+=I_n$, $K_+=K_-= \mathrm{BV}(\Omega)$, which gives  
\begin{equation*}
0\leq -\int_\Omega \lambda_1 f(u_{n-1})(\max(\underline{u}, u_n)-u_n)\, dx.
\end{equation*}
The second inequality is obtained by applying Lemma~\ref{lem:comparison} to $\F_-=I_n$, $\F_+=\F^{[\overline{u}]}_{\overline{\lambda}}$, $K_+=K_-=L^{p+1}\cap \mathrm{BV}(\Omega)$:
\begin{equation*}
0\leq \int_\Omega (\lambda_1 f(u_{n-1})-\overline{\lambda}f(\overline{u}))(\max(\overline{u}, u_n)-\overline{u})\, dx
\end{equation*}
and using the fact that $u_{n-1}\leq \overline u$ and  $\lambda_1<\overline \lambda$.
\smallskip

Since $u_n\in\Lip(\Omega)$, $u_n$ satisfies the Euler-Lagrange equation associated to the minimization of $I_n$: $-\div(T u_n)=H+\lambda_1 f(u_{n-1})$. If $n\geq2$ and $u_{n-1}\geq u_{n-2}$, we then obtain the inequality $u_n> u_{n-1}$ by the strong maximum principle~\refe{strongmax} for Lipschitz continuous functions.
\end{proof}

\begin{proof}[Proof of Lemma \ref{lem:eps}.]
The proof of the lemma is similar to the proof of Lemma~1.1 in~\cite{Giusti78}: 
Assuming that the conclusion is false, we deduce that there exists a sequence $A_k$ of (non-empty) subsets of $\Omega$ satisfying $\ds\int_{A_k}\overline{H}\geq (1-k^{-1})P(A_k)$, $\overline{H}:=H+\lambda_1 f(u_{n-1})$. In particular $P(A_k)=\ds\int_{\RR^N}|D\varphi_{A_k}|$ is bounded, so there exists a Borel subset $A$ of $\Omega$ such that, up to a subsequence, $\varphi_{A_k}\to\varphi_{A}$ in $L^1(\Omega)$ and, by lower semi-continuity of the perimeter, $\ds\int_A\overline{H}\geq P(A)$. This is a contradiction to the strict inequality \refe{Scondition} except if $A$ is empty.
But the isoperimetric inequality gives
$$
|A_k|^{\frac{n}{n-1}}\leq P(A_k)\leq (1-k^{-1})^{-1}\ds\int_{A_k}\overline{H}\leq (1-k^{-1})^{-1}\|\overline{H}\|_{L^n(A_k)}|A_k|^{\frac{n}{n-1}}
$$
hence
$$ 
(1-k^{-1}) \leq \|\overline{H}\|_{L^n(A_k)} \quad \mbox{ for all $k\geq 2$}.
$$
Since $\overline {H}$ is bounded (remember that $u_{n-1}\in\Lip(\Omega)$), we deduce
$$ \frac{1}{2} \leq C |A_k|^{1/n}$$
and so $|A|>0$ since $\varphi_{A_k}\to\varphi_{A}$ in $L^1(\Omega)$.
Consequently, $A$ cannot be empty and we have a contradiction.
\end{proof}


\section{Uniform $L^\infty$ bound for minimal weak solutions}\label{sec:linfty}


The goal of this section is to establish  the $L^\infty$ estimate (Proposition \ref{prop:infty0}) for $\lambda<\lambda^*$. More precisely, we show:
\begin{proposition}\label{prop:infty}
There exists a constant $C$ depending only on $\Omega$ and $H$ such that, for every $0\leq\lambda<\lambda^*$, the minimal weak solution $u_\lambda$ to \refe{eq:1} satisfies 
$$ \|u_\lambda\|_{L^\infty(\Omega)} \leq C.$$
\end{proposition}
This estimate will be used in the next section to show that $u_\lambda$ converges to a weak solution of (\ref{eq:1}) as $\lambda\to\lambda^*$.

The proof relies on an energy method \`a la DeGiorgi~\cite{DeGiorgi57}.
Note that, in general, weak solutions are not minimizers (not even local ones) of the energy functional $\J_\lambda$.
But it is  classical that the minimal solutions $u_\lambda$ enjoy some semi-stability properties.
More precisely, we will show that $u_\lambda$ is a global minimizer of $\J_\lambda$ with respect to non-positive perturbations.
We will then use classical calculus of variation methods to prove Proposition \ref{prop:infty}.

\subsection{Minimal solutions as one-sided global minimizers}
We now show the following lemma:

\begin{lemma}\label{lem:min1} The minimal weak solution  $u_\lambda$ of (\ref{eq:1}) is a global minimizer of the functional $\J_\lambda$ over the set $K_\lambda=\{v\in L^{p+1}\cap\mathrm{BV}(\Omega);0\leq v\leq u_\lambda\}$.
Furthermore, $u_\lambda$ is a semi-stable solution in the sense that, if $u_\lambda\in\Lip(\Omega)$, then $\J_\lambda''(u_\lambda)\geq 0$: for all $\vphi$ in $\mathcal C^1(\Omega)$ satisfying $\vphi=0$ on $\pa\Omega$, we have:
\begin{equation}\label{eq:ss}
Q_\lambda(\vphi):=
\int_\Omega \frac{|\na \vphi|^2}{(1+|\na u_\lambda|^2)^{1/2}} -\frac{|\na \vphi\cdot \na u_\lambda|^2}{(1+|\na u_\lambda|^2)^{3/2}}- \lambda f'(u_\lambda) \vphi^2 \, dx \geq 0.
\end{equation}
\end{lemma}

\begin{proof}
It is readily seen that  the functional $\J_\lambda$ admits a global minimizer $\tilde u_\lambda$ on $K_\lambda$. 
We are going to show that $\tilde u_\lambda=u_\lambda$ by proving, by recursion on $n$, that $\tilde u_\lambda\geq u_n$ for all $n$, where  $(u_n)$ is the sequence used to construct the minimal weak solution $u_\lambda$ in  the proof of Proposition~\ref{prop:min}, that is  $u_0=\underline{u}$ and  $I_n(u_n)=\min_{v\in \mathrm{BV}(\Omega)} I_n(v)$ with, we recall, 
$$I_n(v)=\A(v)-\ds\int_\Omega (H+\lambda f(u_{n-1}))v+\int_{\partial\Omega}|v|d\H^{N-1}.$$ 
\medskip

Set $u_{-1}=0$, so that $u_0=\underline{u}$ is the minimizer of $I_0$. Let $n\geq 0$. Applying Lemma~\ref{lem:comparison} to $\F_-=I_n$, $\F_+=\J_\lambda$, $K_-= \mathrm{BV}(\Omega)$, $K_+=K_\lambda$,  we obtain
\begin{equation}
0\leq \lambda\int_\Omega F(\tilde u_\lambda) - F(\max(u_n,\tilde u_\lambda))+f(u_{n-1})(\max(u_n,\tilde u_\lambda)-\tilde u_\lambda)\, dx.
\label{nVStilde}
\end{equation}
For $n=0$, \refe{nVStilde} reduces to:
\begin{equation*}
0\leq -\int_\Omega  F(\max(\underline {u},\tilde u_\lambda)) - F(\tilde u_\lambda) \, dx,
\end{equation*}
which implies  $\underline{u}\leq\tilde u_\lambda$ a.e. in $\Omega$ since $F$ is increasing.

For $n\geq 1$, assuming that we have proved that $u_{n-1}\leq\tilde u_\lambda$ a.e. in $\Omega$, we have $f(u_{n-1})\leq f(\tilde u_\lambda)$ and \refe{nVStilde} implies 
\begin{align*}
0&\leq -\int_\Omega  F(\max(u_n,\tilde u_\lambda))-F(\tilde u_\lambda)-f(\tilde u_\lambda )(\max(u_n,\tilde u_\lambda)-\tilde u_\lambda)\, dx \\
& =   - \int_\Omega  G(\tilde u_\lambda,\max(u_n,\tilde u_\lambda))\, dx.
\end{align*}
The strict convexity of $F$ implies $\tilde u_\lambda=\max(u_n,\tilde u_\lambda)$ and thus $u_n\leq\tilde u_\lambda$ a.e. in $\Omega$.
\smallskip

Passing to the limit $n\rightarrow \infty$, we deduce $u_\lambda \leq \tilde u_\lambda$ in $\Omega$
and thus $u_\lambda =\tilde u_\lambda$, which completes the proof that $u_\lambda$ is a one sided minimizer.
\medskip

Next, we note that if $\vphi$ is a non-positive smooth function satisfying $\vphi=0$ on $\pa\Omega$, then $\J_\lambda(u_\lambda+t\vphi)\geq \J_\lambda(u_\lambda)$ for all $t\geq 0$. Letting $t$ go to zero, and assuming that $u_\lambda\in\Lip(\Omega)$, we deduce
that the second variation $Q_\lambda(\vphi)$ is non-negative.
Since $Q_\lambda(\vphi)=Q_\lambda(-\vphi)$, it is readily seen that \refe{eq:ss} holds true for non-negative functions. Finally decomposing $\vphi$ into its positive and negative part, we deduce \refe{eq:ss} for any $\vphi$. 
\end{proof}

\subsection{$L^\infty$ estimate}

We now prove:
\begin{proposition}[$L^\infty$ estimate] Let $\lambda\in(0,\lambda^*)$. There exists a constant $C_1$ depending on $\lambda^{-1}$ and $\Omega$ such that 
the minimal weak solution $u_\lambda$ satisfies $\|u_\lambda\|_{L^\infty(\Omega)}\leq C_1.$
\label{prop:L1Linfty}
\end{proposition}
Note that this implies  Proposition~\ref{prop:infty}:  Proposition~\ref{prop:L1Linfty} gives the existence of  $C$ depending only on $\Omega$  such that $\|u_\lambda\|_{L^\infty(\Omega)}\leq C$ for every $\min(1,\lambda^*/2)\leq\lambda<\lambda^*$. And since $0\leq u_\lambda\leq u_{\lambda'}$ if $\lambda<\lambda'$, the inequality is also satisfied when $0\leq\lambda\leq\min(1,\lambda^*/2)$.

\begin{proof}
This proof is essentially a variation of the proof of Theorem~2.2 in Giusti  \cite{Giusti76}.
We fix $\lambda\in(0,\lambda^*)$ and set $u=u_\lambda$.

For some fixed $k>1$, we set $v_k=\min(u,k)$ and $w_k=u-v_k=(u-k)_+$. The difference between the areas of the graphs of $u$ and $v_k$ can be  estimated by below as follows (\cite{Gerhardt74}):
\begin{equation*}
\ds\int_\Omega |Dw_k|-|\{u>k\}|\leq\A(u)-\A(v_k).
\end{equation*}
On the other hand, since $0\leq v_k\leq u$, Lemma~\ref{lem:min1} gives $\J_\lambda(u)\leq \J_\lambda(v_k)$, which implies
\begin{equation*}
\A(u)-\A(v_k)\leq\ds\int_\Omega H(u-v_k)+\lambda [ F(u)-F(v_k)]\, dx.
\end{equation*}
Writing
$$F(u)-F(v_k) = \int_0^1 f(su+(1-s)v_k)\, ds \, (u-v_k),$$
we deduce the following inequality
\begin{equation}
\ds\int_\Omega |Dw_k|\leq|\{u>k\}|+\ds\int_\Omega \left(H+\lambda\int_0^1 f(su+(1-s)v_k)\, ds \right)\, w_k\, dx.
\label{ineqw}\end{equation}
\smallskip

First, we will show that  \refe{ineqw} implies the following estimate:
\begin{equation}
\|u\|_{L^q(\Omega)}\leq C_1(q),
\label{estimLp}\end{equation}
for every $q\in [1,+\infty)$, where $C_1(q)$ depends on $q,\Omega,\lambda^{-1}$.

Indeed, by Lemma~\ref{lem:A}, we have $\ds\int_A H+\lambda f(u) \, dx\leq P(A)$ for all finite perimeter subset $A$ of $\Omega$. We deduce (using the coarea formula):
\begin{align*}
\ds\int_\Omega (H+\lambda f(u))w_k\, dx&=\ds\int_0^{+\infty}\ds\int_{\{w_k>t\}} H+\lambda f(u)\, dx \,dt\\
&\leq\ds\int_0^{+\infty}P(w_k>t) dt\\
&\leq \ds\int_\Omega |Dw_\lambda|.
\end{align*}
So \refe{ineqw} becomes
\begin{equation*}
0\leq |\{u>k\}|-\ds\lambda \int_{\{u\geq k\}} \left[f(u)- \int_0^1 f(su+(1-s)v_k)\, ds\right]w_k\, dx.
\end{equation*}
Since $u\geq 1 $ and $v_k\geq 1$ on $\{u\geq k\}$, and since $f'(s)\geq 1$ for $s\geq 1$,  we have
\begin{align*} 
f(u) & \geq   f(su+(1-s)v_k) + (u-su-(1-s)v_k) \\
& =  f(su+(1-s)v_k) + (1-s)(u-v_k)
\end{align*} 
on $\{u\geq k\}$. We deduce (recall that  $w_k=u-v_k=(u-k)_+$):
\begin{equation*}
\ds\int_\Omega [(u-k)_+]^2\, dx\leq \ds\frac{2}{\lambda}|\{u>k\}|,
\end{equation*}
which implies, in particular, \refe{estimLp} for $q=2$.  
Furthermore, integrating this inequality with respect to $k\in (k',+\infty)$, we get:
\begin{equation*}
\ds\int_\Omega [(u-k)_+]^3\, dx\leq \ds3\cdot \frac{2}{\lambda}\int_\Omega (u-k)_+\, dx,
\end{equation*}
and
by repeated integration  we obtain:
\begin{equation*}
\ds\int_\Omega [(u-k)^+]^{q}\, dx\leq q(q-1)\ds\frac{1}{\lambda}\ds\int_\Omega [(u-k)^+]^{q-2}\, dx
\end{equation*}
for every $q\geq 3$, which implies \refe{estimLp} by induction on $q$.
\bigskip

Note however, that the constant $C_1(q)$ blows up as $q\rightarrow \infty$, and so we cannot obtain the  $L^\infty$ estimate that way.
We thus go back to  \refe{ineqw}: using Poincar\'e's inequality for $\mathrm{BV}(\Omega)$ functions which vanish on $\partial\Omega$ and \refe{ineqw}, we get
\begin{align*}
\|w_k\|_{L^{\frac{n}{n-1}}(\Omega)}&\leq C(\Omega)\ds\int_\Omega|Dw_k|\\
& \leq  C(\Omega)\left(|\{u>k\}|+\ds\int_\Omega \left(H+\lambda f(u)\right)w_k\right)\\
&\leq C(\Omega) \left( |\{u>k\}|+
\|H+\lambda f(u)\|_{L^n(\{w_k>0\})}
\|w_k\|_{L^{\frac{n}{n-1}}(\Omega)} \right)
\end{align*}
Inequality \refe{estimLp} implies in particular that $H+\lambda f(u)\in L^n(\Omega)$ (with bound depending on $\Omega, \lambda^{-1}$), so there exists $\eps>0$ such that $C(\Omega)\|H+\lambda f(u)\|_{L^n(A)}\leq 1/2$ for any subset $A\subset\Omega$ with $|A|<\eps$. Moreover,  Lemma~\ref{lem:A} gives  $\|u\|_{L^1(\Omega)}\leq P(\Omega)/\lambda$ and therefore 
\begin{equation*}
|\{w_k>0\}|=|\{u>k\}|\leq \frac{1}{k}\frac{P(\Omega)}{\lambda}.
\end{equation*}
It follows that there exists $k_0$ depending on $\Omega, \lambda^{-1}$ such that 
$$C(\Omega)\|H+\lambda f(u)\|_{L^n(\{w_k>0\})}\leq 1/2$$
for $k\geq k_0$. For $k\geq k_0$, we deduce
\begin{equation*}
\|w_k\|_{L^{\frac{n}{n-1}}(\Omega)}=\|(u-k)_+\|_{L^{\frac{n}{n-1}}(\Omega)}\leq 2C(\Omega) |\{u>k\}|.
\end{equation*}
Finally, for $k'>k$, we have $1_{|\{u>k'\}}\leq \left(\frac{ (u-k)_+}{k'-k}\right)^{\frac{n}{n-1}}$ and so
$$
|\{u >k'\}|  \leq  \frac{1}{(k'-k)^{\frac{n}{n-1}}} \|(u-k)_+\|_{L^{\frac{n}{n-1}}(\Omega)} ^{\frac{n}{n-1}}
 \leq   \frac{2C(\Omega)}{(k'-k)^{\frac{n}{n-1}}}|\{u>k\}| ^{\frac{n}{n-1}}
$$
which implies, by classical arguments (see \cite{Stampacchia65}) that $|\{u_\lambda >k\}|$ is zero for $k$ large (depending on $|\Omega|$ and $\lambda^{-1}$). The proposition follows.
\end{proof}

\vspace{20pt}
As a consequence, we have:
\begin{corollary}\label{cor:BV}
There exists a constant $C$ depending only on $\Omega$ and $H$ such that
$$ \int_\Omega |Du_\lambda|\leq C.$$
\end{corollary}
\begin{proof}
By Lemma \ref{lem:min} (ii) and Proposition \ref{prop:L1Linfty}, we get:
\begin{align*} 
\mathcal A (u_\lambda) & \leq  \mathcal A(v) - \int_\Omega (H+\lambda f(u_\lambda))v \, dx+ \int_\Omega (H+\lambda f(u_\lambda))u_\lambda\, dx\\
& \leq  \mathcal A(v)+C\int_\Omega |v|\, dx +C
\end{align*}
for any function $v\in L^{p+1}\cap\mathrm{BV}(\Omega)$ such that $v=0$ on $\pa\Omega$. Taking $v=0$, the result follows immediately.
\end{proof}


\section{Existence of the extremal solution}\label{sec:existence2}

We can now complete the proof of Theorem \ref{mainthm}. The only missing piece is the existence of a weak solution for $\lambda=\lambda^*$, which is given by the following proposition:

\begin{proposition}\label{prop:*}
There exists a function $u^*\in L^{p+1}(\Omega)\cap \mathrm{BV}(\Omega)$ such that
\begin{equation*}
u_\lambda \to u^*\quad  \mbox{ in } L^{p+1}(\Omega)\quad \mbox{ as }\lambda\rightarrow\lambda ^*.
\end{equation*} 
Furthermore, $u^*$ is a weak solution of (\ref{eq:1}) for $\lambda=\lambda^*$.
\end{proposition}
\begin{proof}
Recalling that the sequence $u_\lambda$ is non-decreasing with respect to $\lambda$, it is readily seen that Proposition \ref{prop:infty} implies the existence of a function $u^*\in L^\infty(\Omega)$ such that 
\begin{equation*}
\lim_{\lambda\rightarrow \lambda^*} u_\lambda(x) = u^*(x).
\end{equation*}
Furthermore, by the Lebesgue dominated convergence theorem, $u_\lambda$ converges to $u^*$ strongly in $L^q(\Omega)$ for all $q\in[1,\infty)$.

Next, by lower semi-continuity of the area functional $\A(u)$ and Corollary~\ref{cor:BV}, we have
\begin{equation*}
\A(u^*) \leq \liminf_{\lambda\rightarrow \lambda^*}\A(u_\lambda) <\infty.
\end{equation*} 
So, if we write 
\begin{equation*}
\lambda \int F(u_\lambda)\, dx-\lambda^* \int F({u^*})\, dx =(\lambda-\lambda^*) \int F(u_\lambda)\, dx +\lambda^* \int F(u_\lambda)-F({u^*})\, dx,
\end{equation*} 
it is  readily seen that
\begin{equation*}
 \J_{\lambda^*}(u^*)\leq \liminf_{\lambda\rightarrow \lambda^*} \J_\lambda(u_\lambda).
\end{equation*} 

Furthermore, Lemma \ref{lem:min} yields
\begin{equation*}
\J_\lambda(u_\lambda) \leq \J_\lambda(u^*)+\lambda\int_\Omega G(u_\lambda,u^*)\, dx
\end{equation*} 
and so (using the strong $L^{p+1}$ convergence of $u_\lambda$):
\begin{equation*}
 \limsup_{\lambda\rightarrow \lambda^*} \J_\lambda(u_\lambda)\leq  \J_{\lambda^*}(u^*).
 \end{equation*} 
We deduce the convergence of the functionals:
\begin{equation*}
 \J_{\lambda^*}(u^*) =  \lim_{\lambda\rightarrow \lambda^*} \J_\lambda(u_\lambda)
\end{equation*}
which implies in particular that 
\begin{equation*}
\A(u_\lambda)\to\A(u^*)
\end{equation*}
and so $u_\lambda \rightarrow u^*$ in $L^1(\pa\Omega)$. It follows that $u^*$ satisfies the boundary condition $u^*=0$ on $\Omega$.

Finally, using  Lemma \ref{lem:min} again, we have, for any $v\in L^{p+1}\cap \mathrm{BV}(\Omega)$ with $v=0$ on $\pa\Omega$:
\begin{equation*}
 \J_\lambda(u_\lambda)\leq \J_\lambda(v)+\lambda\int_\Omega G(u_\lambda,v)\, dx
\end{equation*}
which yields, as $\lambda\to\lambda^*$:
\begin{equation*}
 \J_{\lambda^*}(u^*)\leq \J_{\lambda^*}(v)+\lambda^*\int_\Omega G(u^*,v)\, dx.
\end{equation*}
for any $v\in L^{p+1}\cap \mathrm{BV}(\Omega)$ with $v=0$ on $\pa\Omega$.
Lemma \ref{lem:min} implies that $u^*$ is a weak solution of (\ref{eq:1}) for $\lambda=\lambda^*$.
\end{proof}

\bigskip


\section{Regularity of the minimal solution in the radial case}\label{sec:reg}
\subsection{Proof of Theorem \ref{thm:rad}}

Throughout this section, we assume that $\Omega=B_R$ and that $H$ depends on $r=|x|$ only. Then, for any rotation $T$ that leaves $B_R$ invariant, we see that the function  $u^T_\lambda(x)=u_\lambda(Tx)$ is a weak solution of \refe{eq:1}, and the minimality of $u_\lambda$ implies
$$ u_\lambda \leq u_\lambda^T \mbox{ in } \Omega. $$
Taking the inverse rotation $T^{-1}$, we get the opposite inequality and so $u^T_\lambda=u_\lambda$, i.e. $u_\lambda$ is radially (or spherically) symmetric.
Furthermore, equation (\ref{eq:1}) reads:
\begin{equation}\label{eq:1rad}
-\frac{1}{r^{n-1}} \frac{d}{dr}\left(\frac{r^{n-1}u_r}{(1+u_r^2)^{1/2}}\right) = H+\lambda f(u).
\end{equation}
or
\begin{equation}\label{eq:1rad2}
-\left[ \frac{u_{rr}}{(1+u_r^2)^{3/2}} + \frac{n-1}{r} \frac{u_r}{(1+u_r^2)^{1/2}} \right] = H+\lambda f(u)
\end{equation}
together with the boundary conditions
\begin{equation*}
u_r(0)=0,\qquad u(R)=0.
\end{equation*}
Note that, by integration of \refe{eq:1rad} over $(0,r)$, $0<r<R$, we obtain
\begin{equation}
\frac{-r^{n-1}u_r(r)}{(1+u_r(r)^2)^{1/2}}=\int_0^r [H+\lambda f(u)] r^{n-1}dr,
\label{signur}\end{equation}
which gives $u_r\leq 0$, provided $u$ is Lipschitz continuous in $\Omega$ at least.
\medskip

It is classical that the solutions of \refe{eq:laplace} can blow up at $r=0$. In our case however, the functions $u_\lambda$ are bounded in $L^\infty$. We deduce the following result:
\begin{lemma}[Bound on the gradient near the origin]\label{lem:regint}
There exists $r_1\in(0,R)$ and $C_1>0$ such that for any $\lambda\in[0,\lambda^*]$, we have
$$ |\na u_\lambda(x)| \leq C_1 \mbox{ for a.a. $x$ such that $|x|\leq r_1$.}$$
\end{lemma}

\begin{proof}
First, we assume that  $u_\lambda$ is smooth.
Then, integrating \refe{eq:1} over $B_r$, we get:
$$ \int_{\pa B_r} \frac{\na u_\lambda \cdot \nu}{\sqrt{1+|\na u_\lambda |^2}} \, dx = \int_{B_r} H+\lambda f(u_\lambda)\, dx.$$
Since $u_\lambda$  is spherically symmetric, this implies:
\begin{equation}\label{eq:PP}
   \frac{ |({u_\lambda})_r| }{\sqrt{1+|({ u_\lambda})_r |^2}} (r) =\frac{1}{P(B_r)} \int_{B_r} H+\lambda f(u_\lambda)\, dx 
\end{equation}
and the $L^\infty$ bound on $u_\lambda$ yields:
$$  \frac{ |({ u_\lambda})_r| }{\sqrt{1+|({ u_\lambda})_r |^2}} (r) \leq C \frac{|B_r|}{P(B_r)} \leq Cr.$$
In particular, there exists $r_1$ such that $Cr\leq 1/2$ for $r\leq r_1$ and so
\begin{equation}\label{eq:bdr0} 
|(u_\lambda)_r|(r) \leq C_1 \qquad \mbox{ for } r\leq r_1.
\end{equation}
Of course, these computations are only possible if we already know that $u_\lambda$ is a classical solution of \refe{eq:1}.
However, it is always possible to perform the above computations with the sequence $(u_n)$ used in the proof of Proposition \ref{prop:min} to construct $u_\lambda$.
In particular, we note that we have $\underline u\leq u_n\leq u_\lambda $ for all $n$ and 
$$ 
-\mbox{div} (Tu_n) = H+\lambda f(u_{n-1}) \mbox{ in } \Omega
$$
so the same proof as above implies that there exists a constant $C$ independent of $n$ or $\lambda$ such that
$$ 
|\na u_n|\leq C_1  \mbox{ for all $x$ such that $|x|\leq r_1$.}$$
The lemma follows by taking the limit $n\to \infty$ (recall that the whole sequence $u_n$ converges in a monotone fashion to $u_\lambda$).
\end{proof}
\vspace{10pt}

\begin{proof}[Proof of Theorem \ref{thm:rad}]
We now want to prove the gradient estimate \refe{eq:gradrad}.
Thanks to Lemma \ref{lem:regint}, we only have to show the result for $r\in[r_1,R]$.
We denote $u^*=u_{\lambda^*}$.
Since $u^* $ is a weak solution of (\ref{eq:1}), Lemma \ref{lem:A} with $A=B_r$ ($r\in[0,R]$) implies
$$ \int_{B_r} H+\lambda^* f(u^*)\, dx \leq P(B_r)$$
and so, using the fact that $u^*\geq u_\lambda\geq \underline u$, we have
$$ \int_{B_r} H+\lambda f(u_\lambda)\, dx \leq P(B_r) -\int_{B_r} (\lambda^*-\lambda) f(u_\lambda)\leq P(B_r) - (\lambda^*-\lambda) \int_{B_r} f(\underline u)\, dx.$$
Hence \refe{eq:PP} becomes:
$$  \frac{ |({ u_\lambda})_r| }{\sqrt{1+|({ u_\lambda})_r |^2}} (r) \leq 1- \frac{(\lambda^*-\lambda)}{P(B_r)} \int_{B_r} f(\underline u)\, dx.$$
For $r\in(r_1,R)$, we have 
$$  \frac{(\lambda^*-\lambda)}{r^{n-1}} \int_{B_r} f(\underline u)\, dx \geq (\lambda^*-\lambda) \delta>0$$
for some universal $\delta$
and so 
$$ 
|(u_\lambda)_r|(r) \leq \frac{C}{\lambda^*-\lambda} \qquad \mbox{ for } r\in [r_1,R].
$$
Together with \refe{eq:bdr0}, this gives the result.
\smallskip

Note once again that these computations can only be performed rigorously on the functions $(u_n)$, which satisfy in particular $\underline u\leq u_n\leq u^*$ for all $n$. So \refe{eq:gradrad} holds for $u_n$ instead of $u_\lambda$. The result follows by passing to the limit $n\to \infty$.
\end{proof}

\begin{remark}
  We point out that the Lipschitz regularity near the origin $r=0$ is a consequence of the $L^\infty$ estimate (it is in fact enough to have $f(u_\lambda)\in L^n$), while the gradient estimate away from the origin only requires $f(u_\lambda)$ to be integrable.
\end{remark}

\subsection{Regularity of the extremal solution}

In this section, we prove Theorem \ref{thm:d2}, that is the regularity of the extremal solution $u^*$.
The proof is divided in two parts: boundary regularity and interior regularity.

\subsubsection{Boundary regularity}
We have the following  a priori estimate:
\begin{lemma}[Bound on the gradient at the boundary]\label{lem:bd2}
Assume that $\Omega=B_R$, that $H$ depends on $r$ only and that conditions \refe{eq:AA}, \refe{eq:BB} and \refe{eq:CC} are fulfilled. 
Let  $u$ be any classical solution of (\ref{eq:1}). 
Then there exists a constant $C$ depending only on  $R$, $\eps_0$ and $n$ such that
$$|u_r (R)| \leq C(1+\lambda).$$
\end{lemma}

Since we know that $u_\lambda\in\Lip(\Omega)$ for $\lambda<\lambda^*$, 
Proposition \ref{prop:cassicalsol} implies that $u_\lambda$ is a classical solution, so Lemma \ref{lem:bd2} yields
$$ |(u_\lambda)_r (R)| \leq C(1+\lambda)\quad\mbox{ for all } \lambda<\lambda^*.$$
Passing to the limit, we obtain:
\begin{equation}\label{eq:nabd}
 | u^*_r(R) |\leq C(1+\lambda^*). 
 \end{equation}

\begin{proof}[Proof of Lemma \ref{lem:bd2}:] 
In this proof, Assumption 
(\ref{eq:BB}) plays a crucial role. When $\Omega$ is a ball of radius $R$ and using the fact that $H\in \Lip(\Omega)$, it implies:
\begin{equation}\label{eq:condHr} 
H(r)\leq (1-\eps_0)\frac{n-1}{R}
\end{equation}
in a neighborhood of $\pa \Omega$ (with a slightly smaller $\eps_0$).
The argument of our proof is similar to the proof of Theorem \ref{thm:gia} (ii) (to show that $u$ satisfies the Dirichlet condition), and relies on the construction of an appropriate barrier.  
Actually, whenever we have  $H(y)\leq (n-1)\Gamma(y)$, $y\in\partial\Omega$, there is a  a natural barrier at the boundary given by the cylinder generated by $\pa B_R$. 
Here, we modify this cylinder by slightly bending it along its generating straight line. The generating straight line thus becomes a circle of radius $\eps^{-1}$ and condition \refe{eq:condHr}  implies that this hypersurface is a supersolution for (\ref{eq:1}).
By radial symmetry, this amounts to consider a circle of radius $\eps^{-1}$ ($\eps$ to be determined) centered at $(M,\delta)$
with $\delta$ small and $M>R$ chosen such that the circle passes through the point $(R,0)$ (see Figure \ref{fig:barrier}).
We define the function  $h(r)$ in  $[M-\eps^{-1},R]$ such that $(r,h(r))$ lies on the circle (with $h(r)<\delta$).

Then, we note that for $r\in[M-\eps^{-1},R]$ and $\eps\delta\leq 1$, we have
$$\frac{h'(r)}{(1+h'(r)^2)^{1/2}} \leq \frac{h'(R)}{(1+h'(R)^2)^{1/2}} = -(1-(\delta\eps)^2)^{1/2}\leq -1+(\delta\eps)^2$$
(this quantity can be interpreted as the horizontal component of the normal vector to the circle),
and 
$$ \frac{d}{dr}\left ( \frac{h'(r)}{(1+h'(r)^2)^{1/2}} \right) =\eps$$
(this quantity is actually the one-dimensional curvature of the curve $r\mapsto h(r)$).
Hence we have:
\begin{align*} 
 \frac{1}{r^{n-1}}\frac{d}{dr}\left ( \frac{r^{n-1} h'(r)}{(1+h'(r)^2)^{1/2}} \right) & =  \frac{d}{dr}\left ( \frac{h'(r)}{(1+h'(r)^2)^{1/2}} \right) +\frac{n-1}{r} \frac{h'(r)}{(1+h'(r)^2)^{1/2}} \\
 &\leq  \eps +\frac{n-1}{r}  ( -1+(\delta\eps)^2) \\
 &\leq  \eps +\frac{n-1}{R}  ( -1+(\delta\eps)^2) 
 \end{align*}

We now use a classical sliding method: Let
$$\eta^*=\inf\{\eta>0 \, ;\, u(r)\leq h(r-\eta) \mbox{ for } r\in[M-\eps^{-1}+\eta,R] \}.$$

If $\eta^*>0$, then $h(r+\eta^*)$ touches $u$ from above at a point in $ (M-\eps^{-1}+\eta,R) $ such that $u< \delta$ (recall that $u$ is Lipschitz continuous so it cannot touch $h(r-\eta)$ at $M-\eps^{-1}+\eta$ since $h=\delta$ and $h'=\infty$ at that point).
At that contact point, we must thus  have
\begin{align*}
\frac{1}{r^{n-1}}\frac{d}{dr}\left ( \frac{r^{n-1} h'(r)}{(1+h'(r)^2)^{1/2}} \right) & \geq 
 \frac{1}{r^{n-1}}\frac{d}{dr}\left ( \frac{r^{n-1} u_r(r)}{(1+u_r(r)^2)^{1/2}} \right) \\
 & \geq   -(H+\lambda f(u))\\
 & \geq  -(1-\eps_0)\frac{n-1}{R} - \lambda\delta^p.
 \end{align*}
We will get a contradiction if $\eps$ and $\delta$ are such that
$$ \eps +\frac{n-1}{R}  ( -1+(\delta\eps)^2) < -(1-\eps_0)\frac{n-1}{R} - \lambda\delta^p$$
which is equivalent to
$$ \eps +\lambda\delta^p+\frac{n-1}{R} (\eps\delta)^2 <\frac{n-1}{R}\eps_0.$$
This can be achieved easily by choosing $\eps$ and $\delta$ small enough.

It follows that $\eta^*=0$ and so $u\leq h$ in the neighborhood of $R$. Since $u(R)=h(R)=0$, we deduce:
$$ |u'(R)| \leq |h'(R)|\leq C(R,n)(\eps\delta)^{-1} \leq C(R,n) \frac{1+\lambda}{\eps_0^2}.$$

\begin{figure}[htbp]
\begin{center}
\scalebox{0.5}{\pdfimage{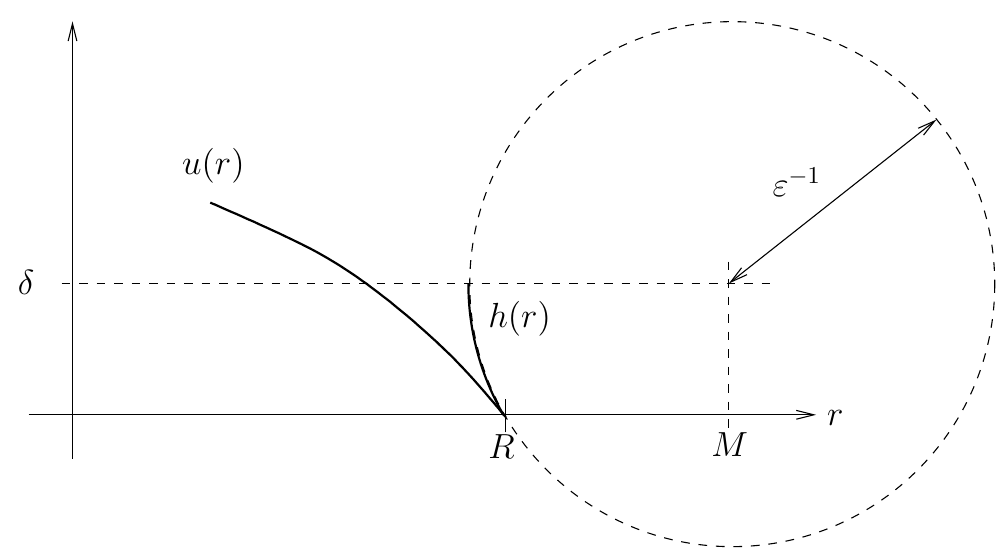} }
\caption{Construction of a barrier}
\label{fig:barrier}
\end{center}
\end{figure}

\end{proof}

\begin{corollary}[Bound on the gradient near the boundary]\label{cor:regbd}
Under the hypotheses of Lemma~\ref{lem:bd2}, there exist $\eta\in(0,R)$ and $C>0$ depending on $R$, $\eps_0$ and $n$ only such that 
$$|u_r (r)| \leq C\qquad \mbox{ for all } r\in [R-\eta,R].$$
\end{corollary}
\begin{proof}
The same proof as that of
Lemma \ref{lem:bd2} shows that there exists $\delta>0$ and $C>0$ such that:
\begin{equation}\label{eq:if}
\mbox{If $u(r)\leq \delta$ for all $r\in [r_0,R]$ with $R-r_0\leq \delta$ then
$ |u_r(r_0)|\leq C.$}
\end{equation}
Furthermore, the proof of Lemma \ref{lem:bd2} implies that $u(r)\leq h(r)$ in a neighborhood of $R$, and so for some small $\eta$ we have:
$$ u(r)\leq \delta \mbox{ for all }  r\in [R-\eta,R].$$
The result follows.
\end{proof}

\subsubsection{Interior regularity}

We now show the following interior regularity result:
\begin{proposition}[Interior bound on the gradient] \label{prop:regmax}
Let $\eta\in(0,R/2)$. There exists $C_\eta>0$ depending only on $\eta$, $n$ and $\int_\Omega|Du_\lambda|$  such that, for all $0\leq\lambda<\lambda^*$, 
$$|\nabla u_\lambda(x)|\leq C_\eta \mbox{ for all $x$ in $\Omega$ with $\eta<|x|<R-\eta$.}$$
\end{proposition}

Using Lemma~\ref{lem:regint} (regularity for $r$ close to $0$), Corollary~\ref{cor:regbd} (regularity for $r$ close to $R$), and Proposition~\ref{prop:regmax} (together with Corollary~\ref{cor:BV} which give the $\mbox{BV}$ estimate uniformly with respect to $\lambda$), we deduce that there exists $C$ depending only on $H$ and $n$ such that
$$ 
|\nabla u_\lambda(x)|\leq C \mbox{ for all $x$ in $\Omega$},
$$
for all $\lambda\in[0,\lambda^*)$.
Theorem \ref{thm:d2} then follows  by passing to the limit $\lambda\to\lambda^*$.
\vspace{10pt}

\begin{proof}[Proof of Proposition \ref{prop:regmax}] 
It is sufficient to prove the result for $\frac{\lambda^*}{2}<\lambda<\lambda^*$.
Throughout the proof, we fix $\lambda\in(\frac{\lambda^*}{2},\lambda^*)$, $r_0\in (\eta,R-\eta)$ and we denote 
$$u=u_\lambda \qquad \mbox{ and } \qquad v=\sqrt{1+u_r^2}.$$
\medskip

\noindent{\bf Idea of the proof:} Let $\vphi_0=\vphi_{B_{r_0}}$ (the characteristic function of the set $B_{r_0}$).
Then by definition of $\J_\lambda$, we have for all $t\geq 0$:
\begin{equation*} 
\J_\lambda (u +t\vphi_0) \leq 
\J_\lambda(u) + t \int_\Omega|D\vphi_0| -t\int_\Omega H\vphi_0\, dx  - \lambda\int_\Omega F(u+t\vphi_0) -F(u)\, dx \end{equation*}
Furthermore, since $u\geq \underline u$, we have $u\geq \mu >0$ in $B_{r_0}$ and so 
$$F(u+t\vphi_0) -F(u) \geq f(u)t\vphi_0 +  \frac{\alpha}{2}t^2 \vphi_0^2 \quad \mbox{ for all }x\in\Omega $$
(with $\alpha$ such that $f'(s)\geq \alpha$ for all $s\geq \mu$).
It follows:
\begin{align*} 
\J_\lambda (u +t\vphi_0) &\leq
\J_\lambda(u) + t \int_\Omega|D\vphi_0| -t\int_\Omega (H+\lambda f(u))\vphi_0\, dx - t^2\frac{\alpha  \lambda}{2} \int_\Omega \vphi_0^2\, dx\\
&= \J_\lambda(u) + t P(B_{r_0}) -t\int_{B_{r_0}} H+\lambda f(u)\, dx - t^2\frac{\alpha  \lambda}{2}  |B_{r_0}|\\
&= \J_\lambda(u)  +tP(B_{r_0})\left(1-\frac{|u_r(r_0)|}{v(r_0)}\right)- t^2\frac{\alpha  \lambda}{2}|B_{r_0}|.
\end{align*}
where we used the following equality, obtained by  integration of (\ref{eq:1}) over  $B_{r_0}$:
\begin{equation*}
-P(B_{r_0})\frac{u_r(r_0)}{v(r_0)}=\int_{B_{r_0}} H+\lambda f(u)\, dx.
\end{equation*}

This would imply $\frac{|u_r|}{v}\leq 1-\delta$ and yield Proposition  \ref{prop:regmax} if we had $\J_\lambda(u) \leq \J_\lambda (u +t\vphi_0)$ for some $t>0$.
Unfortunately, $u=u_\lambda$ is only a minimizer with respect to negative perturbations. 
The proof of Proposition  \ref{prop:regmax} thus consists in using the semi-stability to show that $u$ is almost a minimizer (up to some term of order $3$) with respect to some positive perturbations.

\vspace{20pt}

\noindent{\bf Step 1:}
First of all, the function $\vphi_0$ above is not smooth, so we need to consider the following piecewise linear approximation of $\vphi_0$: 
$$ \vphi_\eps = \left\{ \begin{array}{ll}1 & \mbox{ if }   r\leq r_0-\eps \\ 
\eps^{-1}(r_0-r) &\mbox{ if }  r_0-\eps \leq  r\leq r_0 \\
0 & \mbox{ if }   r\geq r_0.\end{array}\right.$$
We then have (using Equation (\ref{eq:1}) and denoting by $\omega_n$ the volume of the unit ball in $\R^n$):
\begin{align*}
\J_\lambda (u +t\vphi_\eps)
& \leq  \J_\lambda(u) +t \int_\Omega |\nabla \vphi_\eps|\, dx -t\int_\Omega (H+\lambda f(u))\vphi_\eps \, dx-  t^2\frac{\alpha\lambda}{2} \int_\Omega \vphi_\eps^2\, dx \\
& =  \J_\lambda(u) +t \int_\Omega |\nabla\vphi_\eps|\, dx -t\int_\Omega \frac{(u)_r (\vphi_\eps)_r}{v} \, dx- t^2\frac{\alpha\lambda}{2} \int_\Omega \vphi_\eps^2\, dx \\
& =  \J_\lambda(u) +t \int_\Omega \left(1- \frac{|u_r|}{v}\right)|\nabla\vphi_\eps|\, dx -   t^2\frac{\alpha\lambda}{2}\int_\Omega \vphi_\eps^2\, dx \\
& =  \J_\lambda(u) +t \omega_n \int_{r_0-\eps}^{r_0} \left(1- \frac{|u_r|}{v}\right)\eps^{-1}r^{n-1}\, dr - t^2\frac{\alpha\lambda}{2}  \int_\Omega \vphi_\eps^2\, dx  \\
& \leq  \J_\lambda(u) +t \omega_n \eps^{-1}\int_{r_0-\eps}^{r_0}  \frac{1}{v^2}r^{n-1}\, dr -  t^2\frac{\alpha\lambda}{2} \int_\Omega \vphi_\eps^2\, dx 
\end{align*}
and so if we denote $\rho(\eps) = \sup_{r\in (r_0-\eps,r_0)} \frac{1}{v^2}$, we deduce:
\begin{equation} \label{eq:te}
\J_\lambda (u +t\vphi_\eps)
 \leq  \J(u)_\lambda +t \omega_n  r_0^{n-1}\rho(\eps) - t^2\frac{\alpha\lambda}{2} \omega_n \left(\frac{r_0}{2}\right)^n 
\end{equation}
 for all  $\eps<r_0/2$.

\vspace{20pt}

\noindent{\bf Step 2:}
Since our goal is to show that $\rho(\eps)$ is cannot be too small, we need to control $\J (u +t\vphi_\eps)$ from below:
for a smooth radial function $\vphi$, we denote
$$\theta(t)=\mathcal A(u+t\vphi)=\int_\Omega L(u_r+t\varphi_r),$$
where $L(s)=(1+s^2)^{1/2}$.
Then
$$\theta^{(3)}(t) = \int_\Omega L^{(3)}(u_r+t\varphi_r)\varphi_r^3\, dx$$
where
$$L^{(3)}\colon s\mapsto\frac{-3s}{(1+s^2)^{5/2}}$$ 
satisfies
$$| L^{(3)} (s)| \leq \frac{3}{(1+s^2)^{2}},\quad\forall s\geq 0.$$

When $\vphi=\vphi_\eps$, we have  $|u_r+t\varphi_r | \geq |u_r| $ for all $t\geq 0$ and therefore:
\begin{align*}
|\theta^{(3)}(t)|  & \leq   \int_\Omega \frac{3}{v^4}(|(\varphi_\eps)_r|^3\, dx\\
& \leq   \eps^{-3}\omega_n \int_{r_0-\eps}^{r_0} \frac{3}{v^4} r^{n-1}\, dr\\
& \leq   \eps^{-2}\omega_n \, \rho (\eps)^2 \, r_0^{n-1}
\end{align*}
for all $t\geq 0$.

Since the second variation $Q_\lambda(\vphi_\eps)$ is non-negative by Lemma \ref{lem:min1} (recall that $u_\lambda$ is a semi-stable solution), we deduce that for some $t_0\in(0,t)$ we have:
\begin{align}
\J_\lambda (u +t\vphi_\eps) & =  \J_\lambda(u) + \frac{t^2}{2} Q_\lambda(\vphi_\eps) + \theta^{(3)}(t_0) \frac{t^3}{6} - \lambda \int_\Omega \frac{f''(u+t_0\vphi_\eps)}{6} t^3 \vphi^3\, dx   \nonumber\\
& \geq   \J_\lambda(u) -\frac{t^3}{2} | \theta^{(3)}(t_0)|-\|f''(u+t_0\vphi_\eps)\|_{L^\infty(B_{r_0})}\lambda t^3\omega_n r_0^n\nonumber \\
&\geq    \J_\lambda(u)- \frac{t^3}{2}\eps^{-2}\omega_n \,\rho (\eps)^2  \, r_0^{n-1}  -C \lambda t^3\omega_n r_0^n,\label{eq:thirdvar}
\end{align}
where we used the fact that $f''(u+t_0\vphi_\eps)\in L^\infty(B_{r_0})$ (if $p\geq 2$, this is a consequence of the $L^\infty$ bound on $u$, if $p\in (1,2)$, then this follows from the fact that $u+t_0\vphi_\eps\geq \underline u >0$ in $B_{r_0}$).

\vspace{20pt}

\noindent{\bf Step 3:}
Inequalities (\ref{eq:te}) and (\ref{eq:thirdvar}) yield:
$$
\lambda \frac{t^2}{2} \omega_n r_0^n  \leq t \omega_n  r_0^{n-1}\rho(\eps) +\frac{t^3}{2} \eps^{-2}\omega_n \,\rho (\eps)^2  \, r_0^{n-1}+C \lambda t^3\omega_n r_0^n
$$
and so 
$$
\frac{\lambda r_0}{2}(1-2C t) t  \leq \rho(\eps) + \frac{\eps^{-2}t^2}{2}  \,\rho (\eps)^2 
$$
for all $t\geq 0$.
If $t\leq 1/(4C )$, we deduce
$$
\mu t  \leq \rho(\eps) + \frac{\eps^{-2}t^2}{2}  \,\rho (\eps)^2 
$$
with $\mu=\lambda r_1/4$ (recall that $r_0>r_1$).

Let now $t=M\eps$ ($M$ to be chosen later), then we get
$$
\mu M\eps  \leq \rho(\eps) + \frac{M^2}{2}  \,\rho (\eps)^2 .
$$
If $\rho(\eps)\leq \frac{\mu M\eps}{2}$, then 
$$  \rho(\eps) + \frac{M^2}{2}  \,\rho (\eps)^2 \leq \frac{\mu M\eps}{2} + \frac{\mu^2 M^4 \eps^2}{8}
$$
and we get a contradiction if $\frac{\mu^2 M^4 \eps^2}{8} < \frac{\mu M\eps}{2}$. It follows that 
\begin{equation}\label{eq:rho}  
\rho(\eps) \geq \frac{\mu M\eps}{2} \qquad \mbox{ for all } \eps< \frac{4}{\mu M^3}.
\end{equation}

\vspace{20pt}

\noindent{\bf Step 4:}
Since  $\rho(\eps) = \sup_{r\in (r_0-\eps,r_0)} \frac{1}{v^2}$,   (\ref{eq:rho}) yields
$$ 
\inf _{r\in (r_0-\eps,r_0)} v^2 \leq \frac{2}{\mu M\eps}
\qquad \mbox{ for all } \eps< \frac{4}{\mu M^3}.
$$
In order to conclude, we need to use some type of Harnack inequality to control $\sup _{r\in (r_0-\eps,r_0)} v^2 $.
This will follow from the following lemma: 
\begin{lemma}\label{lem:vv}
Let $v=\sqrt{1+u_r^2}$. Then $v$ solves the following equation in $(0,R)$:
\begin{equation}\label{eq:vv}
-\frac{1}{r^{n-1}}\left(\frac{r^{n-1}v_r}{v^3}\right)_r + c^2 =H_r \frac{u_r}{v}+\lambda f'(u)\frac{u_r^2}{v}.
\end{equation}
where
$$c^2 =  \frac{n-1}{r^2}\frac{u_r^2}{v^2}+ \frac{u_{rr}^2}{v^6}$$
is the sum of the square of the curvatures of the graph of $u$.
\end{lemma}
We postpone the proof of this lemma to the end of this section.
Clearly, the equation (\ref{eq:vv}) is degenerate elliptic. In order to write a Harnack inequality, we introduce $w=\frac{1}{v^2}$, solution of the following equation
$$ 
\frac{1}{r^{n-1}}\left(r^{n-1}w_r\right)_r  =2H_r \frac{u_r}{v}+2\lambda f'(u)\frac{u_r^2}{v}-2c^2
$$
which is a nice uniformly elliptic equation in a neighborhood of $r_0\in(0,R)$.
In particular, if $\eps\leq R-r_0$,  Harnack's inequality \cite{GT} yields:
\begin{equation}\label{eq:harnack}
\sup_{r\in (r_0-\eps,r_0)} w \leq C \inf_{r\in (r_0-\eps,r_0)} w + C \eps \|g\|_{L^1(r_0-2\eps,r_0+\eps)} 
\end{equation}
where 
$$ g = 2H_r \frac{u_r}{v}+2\lambda f'(u)\frac{u_r^2}{v}-2c^2.$$

Next, we note that 
$$|g|\leq 2|H_r|+C\lambda |u_r| + 2 c^2.$$
It is readily seen that the first $(n-1)$ curvatures $\frac{1}{r}\frac{u_r}{v}$ are bounded in a neighborhood of $r_0\neq 0$. Furthermore, since the mean curvature is in $L^\infty$, it is easy to check that the last curvature is also bounded: more precisely, (\ref{eq:1rad2}) gives
$$ \frac{u_{rr}}{v^3} = -H-\lambda f(u) - \frac{n-1}{r}\frac{u_r}{v} \in L^\infty.$$
We deduce that $c^2\in L^\infty$ and since $u\in \mathrm{BV}(\Omega)$, we get 
$$ \|g\|_{L^1(r_0-2\eps,r_0+\eps)} \leq C\int_\Omega|Du| +C$$
Together with (\ref{eq:harnack}) and (\ref{eq:rho}) (and recalling that $\rho(\eps) = \sup_{r\in (r_0-\eps,r_0)} w^2$), we deduce:
$$ 
\frac{\mu M\eps}{2}  \leq C \inf_{r\in (r_0-\eps,r_0)} w + C\left( \int_\Omega|Du| +1\right)\eps 
\qquad \mbox{ for all } \eps< \frac{4}{\mu M^3}.
$$
With $M$ large enough ($M\geq \frac{4C}{\mu} \left(\int_\Omega|Du|+1\right)$), it follows that
$$ 
\frac{\mu M\eps}{4}  \leq C \inf_{r\in (r_0-\eps,r_0)} w \qquad \mbox{ for all } \eps< \frac{4}{\mu M^3}
$$
and thus (with $\eps=\min(\frac{2}{\mu M^3},(R-r_0)/4,\frac{1}{4MC}$)):
$$ v(r_0)^2 \leq \sup_{r\in (r_0-\eps,r_0)} v^2 \leq C((\lambda r_0)^{-1},(R-r_0)^{-1},\int_\Omega|Du|,\|u\|_{L^\infty(\Omega)}),$$
which completes the proof.
\end{proof}

\begin{proof}[Proof of Lemma \ref{lem:vv}]
Taking the derivative of (\ref{eq:1rad})  with respect to $r$ and multiplying by $u_r $, we get:
$$ 
\frac{n-1}{r^n}\left(\frac{r^{n-1}u_r}{v}\right)_r u_r-\frac{1}{r^{n-1}}\left(\frac{r^{n-1}u_r}{v}\right)_{rr} u_r = H_r u_r+ \lambda f'(u) u_r^2. 
$$
Using the fact that
$$\left(\frac{u_r}{v}\right)_r = \frac{u_{rr}}{v^3} \qquad \mbox{ and }\qquad v_r = \frac{u_r u_{rr}}{v},$$
we deduce:
\begin{multline*}
\frac{(n-1)^2}{r^n}\frac{r^{n-2}u_r^2}{v}+\frac{n-1}{r}\frac{u_ru_{rr}}{v^3}-\frac{n-1}{r^{n-1}}\left(\frac{r^{n-2}u_r}{v}\right)_r u_r -\frac{1}{r^{n-1}} \left(\frac{r^{n-1}u_{rr}}{v^3}\right)_r u_r \\
=H_r u_r+ \lambda f'(u) u_r^2
\end{multline*}
and so (simplifying and dividing by $v$):
$$ 
\frac{(n-1)^2}{r^2}\frac{u_r^2}{v^2} - \frac{(n-1)(n-2)}{r^{n-1}}\frac{r^{n-3}u_r^2}{v^2} -\frac{1}{r^{n-1}}  \left(\frac{r^{n-1}u_{rr}}{v^3}\right)_r \frac{u_r}{v} =H_r \frac{u_r}{v}+\lambda f'(u)\frac{u_r^2}{v}.
$$
This yields
$$
\frac{(n-1)}{r^2}\frac{u_r^2}{v^2} 
-\frac{1}{r^{n-1}}  \left(\frac{r^{n-1}u_{rr}u_r}{v^4}\right)_r +
\frac{1}{r^{n-1}}  \frac{r^{n-1}u_{rr}}{v^3}\left(\frac{u_r}{v} \right)_r 
=H_r \frac{u_r}{v}+\lambda f'(u)\frac{u_r^2}{v}
$$
hence
$$
\frac{(n-1)}{r^2}\frac{u_r^2}{v^2} 
-\frac{1}{r^{n-1}}  \left(\frac{r^{n-1}v_r}{v^3}\right)_r +
 \frac{u_{rr}^2}{v^6} 
=H_r \frac{u_r}{v}+\lambda f'(u)\frac{u_r^2}{v}
$$
which is the desired equation.
\end{proof}

\subsubsection{Proof of Theorem~\ref{thm:CR}}\label{sec:thmCR}

In this section, we adapt the continuation method of \cite{CR} to prove Theorem~\ref{thm:CR}.

First, we need to introduce some notations:
Let $\alpha\in(0,1)$ and, for $k\in\N$, let $\mathcal C^{k,\alpha}_0(\overline{\Omega})$ be the set of functions $u\in\mathcal  C^{k,\alpha}(\overline{\Omega})$ that satisfy $u=0$ on $\partial\Omega$. Let $\T\colon \mathcal C^{2,\alpha}_0(\overline{\Omega})\times\R\to \mathcal C^{\alpha}_0(\overline{\Omega}) $ be defined by
\begin{equation*}
\T(u,\lambda)=-\div(Tu)-H-\lambda f(u).
\end{equation*}
The function $\T$ is twice continuously differentiable and, at any point $(u,\lambda)\in \mathcal C^{2,\alpha}_0(\overline{\Omega})\times\R$, has first derivatives
\begin{equation*}
\T_u(u,\lambda)\colon v\mapsto -\partial_i(a^{ij}(\nabla u)\partial_j v)-\lambda f'(u)v,\quad \T_\lambda(u,\lambda)=-f(u),
\end{equation*}
where we use the convention of summation over repeated indices and set, for $\p\in\R^n$,
\begin{equation*}
a^i(\p)=\frac{\p_i}{(1+|\p|^2)^{1/2}},\quad a^{ij}(\p)=\frac{\partial a^i}{\partial \p_j}(\p)=\frac{1}{(1+|\p|^2)^{1/2}}\left(\delta_{ij}-\frac{\p_i\p_j}{1+|\p|^2}\right).
\end{equation*}
The second derivatives of $\T$ at any point $(u,\lambda)\in \mathcal C^{2,\alpha}_0(\overline{\Omega})\times\R$ are 
\begin{equation*}
\T_{uu}(u,\lambda)(v,w)=-\partial_i(a^{ijk}(\nabla u)\partial_j v\partial_k w)-\lambda f''(u)vw
\label{secondderivF}\end{equation*}
and $\T_{u\lambda}(u,\lambda)(v,\mu)=-\mu f'(u)v$, $\T_{\lambda\lambda}(u,\lambda)=0$, where
\begin{equation*}
a^{ijk}(\p)=\frac{\partial a^{ij}}{\partial \p_k}(\p)=3\frac{\p_i \p_j \p_k}{(1+|\p|^2)^{5/2}}-\frac{1}{(1+|\p|^2)^{3/2}}(\delta_{ij}\p_k+\delta_{ik}\p_j+\delta_{kj}\p_i).
\label{aijk}\end{equation*}
We note for further use that, given $\p,\q\in\R^n$,
\begin{equation*}
a^{ijk}(\p)\q_i \q_j \q_k=3\frac{\p\cdot \q}{(1+|\p|^2)^{5/2}}((\p\cdot \q)^2-|\q|^2(1+|\p|^2)),
\end{equation*}
and thus, in particular,
\begin{equation}
\p\cdot \q\geq 0\quad \Longrightarrow \quad a^{ijk}(\p)\q_i \q_j \q_k\leq 0.
\label{signaijk}\end{equation}

Next, we note that for any $u,v,w$  radially symmetric function,  non-increasing with respect to $r$, we have
\begin{equation}
[a^i(\nabla u)-a^i(\nabla v)-a^{ij}(\nabla u)\partial_j(u-v)]\partial_i w\geq 0,
\label{convexTradial}\end{equation}
or, equivalently, setting $A(\nabla u):=(a^{ij}(\nabla u))_{ij}$:
\begin{equation}
(Tu-Tv-A(\nabla u)\nabla(u-v))\cdot\nabla w\geq 0.
\label{convexTradial2}\end{equation}
Indeed, the left-hand side of \refe{convexTradial} rewrites
\begin{equation*}
(h(p)-h(q)-h'(p)(p-q))s,\quad\mbox{ where } h(p)=\frac{p}{(1+p^2)^{1/2}}, 
\end{equation*}
where $p=\partial_r u \leq0$, $q=\partial_r v\leq0$, $s=\partial_r w\leq 0$ and $h$ is convex on $\R_-$.
\medskip

Recall that $\underline{u}\in \mathcal C^{2,\alpha}_0(\overline{\Omega})$ is the solution to $\T(\underline{u},0)=0$. In particular $\underline{u}$ is radially symmetric and non-increasing with respect to $r$. At $\lambda=0$, the map $\T_u(\underline{u},0)\colon\mathcal C^{2,\alpha}_0(\overline{\Omega})\to\mathcal C^{\alpha}_0(\overline{\Omega}) $ is invertible since it defines a uniformly elliptic operator with no zero-th order terms. By the Implicit Function Theorem, we obtain the existence of $a>0$ and of a $\mathcal C^2$ curve $\lambda\mapsto u(\lambda)$ from $[0,a]$ to $\mathcal C^{2,\alpha}_0(\overline{\Omega})$ of solutions to $\T(u(\lambda),\lambda)=0$ such that $u(0)=\underline{u}$. 

Let now $\overline{\lambda}\in(0,\lambda^*]$ be the largest $b>0$ such that this curve can be continued to $[0,b)$ under the additional constraint that for all $\lambda\in [0,b)$, $\T_u(u(\lambda),\lambda)$ is invertible. We denote by $L_\lambda$ the elliptic operator $L_\lambda=\T_u(u(\lambda),\lambda)$ and by 
\begin{equation*}
\mu_1(\lambda)<\mu_2(\lambda)\leq\mu_3(\lambda)\ldots
\end{equation*} 
its eigenvalues. It is readily seen that $\mu_1(0)>0$ (since there are no zero-th order terms in $L_0$).
Since $\lambda\mapsto\mu_1(\lambda)$ is continuous\footnote{this follows from the continuity of the map $\lambda\mapsto u(\lambda)$ valued in $\mathcal C^{2,\alpha}(\overline{\Omega})$ and from the characterization of $\mu_1(\lambda)$ as the supremum over non-trivial $\varphi\in\mathcal C^2(\overline{\Omega})$ of the Rayleigh quotients $\frac{(L_\lambda\varphi,\varphi)}{(\varphi,\varphi)}$ where $(\cdot,\cdot)$ is the canonical scalar product over $L^2(\Omega)$}
and $\mu_1(\lambda)\neq 0$ on $[0,\overline{\lambda})$, we see that
$\mu_1(\lambda)>0$ for all $\lambda\in [0,\overline{\lambda})$. 

Note also that the function $u(\lambda)$ is a radially symmetric\footnote{this is the case of every terms in the iterative sequence $u_n(\lambda)$ converging to $u(\lambda)$ that is constructed by application of the Implicit Function Theorem}, and that $L_\lambda$ therefore admits a first eigenvector $w^1_\lambda>0$ associated to the eigenvalue $\mu_1(\lambda)$ which is also  a radially symmetric function. Furthermore, one can check that $w^1_\lambda$ is non-increasing with respect to $r$: As in \refe{eq:1rad}-\refe{signur}, this follows directly from the equation $L_\lambda w^1_\lambda =\mu_1({\lambda}) w^1_\lambda$ written in terms of the $r$-variable, {\it i.e.}
\begin{equation*}
-\frac{1}{r^{n-1}}\partial_r\left(\frac{r^{n-1}}{(1+|\partial_r u(\lambda)|^2)^{3/2}}\partial_r w^1_\lambda\right)=\lambda f'(u(\lambda))w^1_\lambda+\mu_1({\lambda}) w^1_\lambda\geq 0. 
\end{equation*}
We can now prove that $u(\lambda)$ and $u_\lambda$ coincide.

\begin{lemma} We have $\overline{\lambda}=\lambda^*$, $u(\lambda)=u_\lambda$ (the minimal solution), $\mu_1(\lambda)>0$ for all $\lambda\in [0,\lambda^*)$ and $\mu_1(\lambda^*)=0$.
\label{lem:minimalbranch}\end{lemma}

\begin{proof} We adapt the proof of Theorem 3.2 in \cite{KK}. Let $\lambda\in [0,\overline{\lambda})$, $\nu\in[0,\lambda^*]$. 
Using the fact that  $u(\lambda)$ and $u_\nu$ are solutions to \refe{eq:1}, we get:
\begin{align*}
L_\lambda(u(\lambda)-u_\nu)=&-\div(A(\nabla u(\lambda))\nabla(u(\lambda)-u_\nu))-\lambda f'(u(\lambda))(u(\lambda)-u_\nu)\\
=&\lambda f(u(\lambda))-\nu f(u_\nu)-\lambda f'(u(\lambda))(u(\lambda)-u_\nu)\\
&+\div[Tu(\lambda)-Tu_\nu-A(\nabla u(\lambda))\nabla(u(\lambda)-u_\nu)].
\end{align*}
 Since $f$ is convex, we have
\begin{equation*}
\lambda f(u(\lambda))-\nu f(u_\nu)-\lambda f'(u(\lambda))(u(\lambda)-u_\nu)\leq (\lambda-\nu)f(u_\nu)
\end{equation*}
and it follows from \refe{convexTradial} that
\begin{equation}
\int_\Omega L_\lambda(u(\lambda)-u_\nu) w dx\leq (\lambda-\nu)\int_\Omega f(u_\nu) w dx
\label{subtractTradial}\end{equation}
for any radially symmetric non-negative non-increasing function $w\in\mathcal C^{2,\alpha}(\Omega)$. Taking $\nu=\lambda$ and $w=w^1_\lambda$, the positive eigenvector corresponding to the first eigenvalue $\mu_1(\lambda)$,  we deduce:
\begin{equation*}
\mu_1(\lambda)\int_\Omega (u(\lambda)-u_\lambda)w^1_\lambda dx\leq 0.
\end{equation*}
We have $u(\lambda)-u_\lambda\geq 0$ since $u_\lambda$ is the minimal solution to \refe{eq:1} and $\mu_1(\lambda)>0$, $w^1_\lambda>0$ in $\Omega$, hence $u(\lambda)=u_\lambda$ in $\Omega$. 

We now extend the definition of $L_\lambda$ to the whole interval $[0,\lambda^*]$ by setting $L_\lambda=\T_u(u_\lambda,\lambda)$. In particular, $\mu_1(\overline{\lambda})=0$ and \refe{subtractTradial} is valid for $\lambda$ in the whole range $[0,\lambda^*]$. To prove the second part of Lemma~\ref{lem:minimalbranch}, assume by contradiction $\overline{\lambda}<\nu<\lambda^*$. Taking $\lambda=\overline{\lambda}$ and $w=w^1_{\overline{\lambda}}$ in \refe{subtractTradial}, we obtain 
\begin{equation*}
0\leq (\overline{\lambda}-\nu)\int_\Omega f(u_\nu)w^1_{\overline{\lambda}} dx.
\end{equation*}
This is impossible since $\overline{\lambda}<\nu$ and $\int_\Omega f(u_\nu)w^1_{\overline{\lambda}} dx>0$.
\end{proof}

We can now complete the proof of Theorem~\ref{thm:CR}. Let $w_*\in \mathcal C^{2,\alpha}_0(\overline{\Omega})$ be the first eigenvector of $L_{\lambda^*}$: $L_{\lambda^*}w_*=0$, $w_*>0$ in $\Omega$, $w_*$ is radial non-increasing with respect to $r$. 
Let $Z\subset \mathcal C^{2,\alpha}_0(\overline{\Omega})$ be the closed subspace of elements $z\in \mathcal C^{2,\alpha}_0(\overline{\Omega})$ orthogonal (for the $L^2(\Omega)$ scalar product) to $w_*$. Let $\T^*$ be the $\mathcal C^2$ map $\R\times Z\times\R\to \mathcal C^{\alpha}_0(\overline{\Omega})$ defined by 
\begin{equation*}
\T^*(s,z,\lambda)=\T(u_*+s w_* +z,\lambda).
\end{equation*} 
The derivative $\T^*_{z,\lambda}(0,0,\lambda_*)$ is invertible. Indeed, given $v\in \mathcal C^{\alpha}_0(\overline{\Omega})$, $(\zeta,\mu)\in Z\times \R$ is solution to 
\begin{equation*}
\T^*_{z,\lambda}(0,0,\lambda_*)\cdot(\zeta,\mu)=v
\end{equation*}
if
\begin{equation}
L_{\lambda^*}\zeta+\mu f(u_*)=v.
\label{eqinvstar}\end{equation}
By the Fredholm alternative (and the Schauder regularity theory for elliptic PDEs), \refe{eqinvstar} has a unique solution $\zeta\in Z$ provided 
\begin{equation*}
\mu\int_\Omega f(u_*) w_* dx=\int_\Omega v w_* dx.
\end{equation*}
This condition uniquely determines $\mu$ since $f(u_*),w_*>0$ in $\Omega$ and, in particular, $\int_\Omega f(u_*) w_* dx>0$.
By the Implicit Function Theorem, it follows that there is an $\eps>0$ and a $\mathcal C^2$-curve $(-\eps,\eps)\to Z\times\R$, $s \mapsto (z(s),\lambda(s))$ such that 
\begin{equation}
(z,\lambda)(0)=(0,\lambda^*),\quad \T^*(s,z(s),\lambda(s))=0,\quad\forall |s|<\eps.
\label{IFeq}\end{equation}
By derivating once with respect to $s$ in \refe{IFeq}, we obtain
\begin{equation}
L_{\lambda^*}w_*+\T^*_{z,\lambda}(0,0,\lambda^*)\cdot(z'(0),\lambda'(0))=0,
\label{IFonce}\end{equation}
hence $z'(0)=0$, $\lambda'(0)=0$. We set $u(s)=u_*+sw_*+z(s)$. Then $u'(0)=w_*>0$ in $\Omega$. To show the effective bending of the curve $s\mapsto (u(s),\lambda(s))$, there remains to prove that $\lambda''(0)<0$. Let us differentiate twice with respect to $s$ in \refe{IFeq}: we obtain
\begin{multline*}
-\partial_i(a^{ij}(\nabla u)\partial_j u'')-\lambda f'(u)u''-\partial_i(a^{ijk}(\nabla u)\partial_k u'\partial_j u')\\
=\lambda'' f(u)+2\lambda' f'(u) u'+\lambda f''(u)|u'|^2.
\end{multline*}
At $s=0$, this gives 
\begin{equation*}
L_{\lambda^*}u''(0)-\partial_i(a^{ijk}(\nabla u)\partial_k w_*\partial_j w_*)
=\lambda''(0) f(u_*)+\lambda^* f''(u_*)|w_*|^2.
\end{equation*}
Integrating the result against $w_*$ over $\Omega$, we deduce that
\begin{equation}
\int_\Omega a^{ijk}(\nabla u)\partial_k w_*\partial_j w_* \partial_i w_* dx =\lambda''(0)\int_\Omega f(u_*)w_* dx+\lambda^* \int_\Omega f''(u_*)w_*^3 dx.
\label{signlambdasecond}\end{equation}
Since $\nabla u\cdot\nabla w_*=\partial_r u\partial_r w_*\geq 0$, \refe{signaijk} shows that the left-hand side in \refe{signlambdasecond} is non-positive.
Finally, since $f(u_*),f''(u_*),w_*>0$, we get $\lambda''(0)<0$. \qed


\appendix


\section{Comparison principles}


It is well known that classical solutions of (\ref{eq:1}) satisfy a strong comparison principle, namely, if $u,v\in\Lip(\Omega)$ satisfy
\begin{equation}
-\div(Tu)\leq-\div(Tv)\mbox{  in } \Omega , \qquad u\leq v \mbox{ on } \partial\Omega
\label{eq:comparisonprinciple}\end{equation}
with $u\neq v$, then
\begin{equation}
u<v\mbox{ in }\Omega.
\label{strongmax}\end{equation} 
If $u,v$ are in  $W^{1,1}(\Omega)$ and satisfy \refe{eq:comparisonprinciple}, then we still have a weak comparison principle, {\it i.e.} $u\leq v$ a.e. in $\Omega$ (see \cite{GiustiBook}).
But no such principle holds for functions that are only  in $\mathrm{BV}(\Omega)$ (even if one of the function is smooth). 
This is due to the lack of strict convexity of the functional $\A$ on $\mathrm{BV}(\Omega)$ that is affine on any interval $[0,\varphi_A]$ (in particular, we have ${\mathcal L} (\varphi_A)={\mathcal L} (-\varphi_A)={\mathcal L} (0)=0$ for any finite perimeter set $A$).
\medskip

Throughout the paper, we consider weak solutions to \refe{eq:1} which are, {\it a priori}, not better (with respect to integrability properties of the gradient) than $\mathrm{BV}(\Omega)$.  
In order to derive comparison results, we use Lemma~\ref{lem:min}, which allows us to interpret weak solutions as global minimizers of an accurate functional and the following lemma.

\begin{lemma}[Comparison principle] Let $q\geq 1$. Let $G_\pm\colon\Omega\times\R\to\R$ satisfy the growth condition $|G_\pm(x,s)|\leq C_1(x)|s|^q+C_2(x)$ where $C_1\in L^{\infty}(\Omega)$ and $C_2\in L^1(\Omega)$. Let $\F_\pm$ be the functional defined on $L^q\cap \mathrm{BV}(\Omega)$ by 
\begin{equation*}
\F_\pm(v)=\A(v)+\int_{\partial\Omega}|v|d\H^{N-1}+\int_\Omega G_\pm(x,v)\, dx.
\end{equation*}
Suppose that $u_\pm$ is a global minimizer of $\F_\pm$ on a set $K_\pm$ and suppose that 
\begin{equation*}
\min(u_+,u_-)\in K_-,\quad \max(u_+, u_-) \in K_+,
\end{equation*}
Then we have
\begin{equation*}
0\leq\Delta(\max(u_+, u_-))- \Delta(u_+),\quad\Delta(v):=\int_\Omega G_+(x,v)-G_-(x,v)\, dx.
\end{equation*}
\label{lem:comparison}
\end{lemma}

\begin{proof}[Proof of Lemma~\ref{lem:comparison}] 
We need to recall the inequality
\begin{equation}
\int_{Q} |D\varphi_{E\cup F}|+\int_{Q} |D\varphi_{E\cap F}|\leq \int_{Q} |D\varphi_{E}|+\int_{Q} |D\varphi_{F}|,
\label{capcup}\end{equation}
which holds for any open set $Q\subset \R^m$ ($m\geq 1$) and any sets $E,F$ with locally finite perimeter in $\R^m$. 
Applied to $Q=\Omega\times\R$ and to the characteristic functions of the subgraphs of $u$ and $v$, Inequality~\refe{capcup} gives:
\begin{equation}
\A(\max(u, v))+\A(\min(u, v))\leq \A(u)+\A(v),\quad u,v\in \mathrm{BV}(\Omega). 
\label{capcuparea}\end{equation}
Since $\int_\Omega |Du|\leq\A(u)$, this shows in particular that $\max(u, v)$, $\min(u, v)$ and $(u-v)_+=\max(u,v)-v=u-\min(u, v)\in \mathrm{BV}(\Omega)$ whenever $u$ and $v\in \mathrm{BV}(\Omega)$.  
\medskip

Since $u\mapsto \int_\Omega G_\pm(u)$ is invariant by rearrangement, we deduce:
\begin{equation}\label{eq:comp2}
\F_-(\max(u_+, u_-))+\F_-(\min(u_+, u_-))\leq \F_-(u_+)+\F_-(u_-).
\end{equation}
Furthermore, we have $\min(u_+, u_-)\in K_-$, and so
$\F_-(u_-)\leq \F_-(\min(u_+, u_-))$.
Therefore, (\ref{eq:comp2})  implies that $\F_-(\max(u_+, u_-))\leq \F_-(u_+)$, which, by definition of $\Delta$ also reads:
$$ \F_+(\max(u_+, u_-))-\Delta(\max(u_+, u_-))\leq \F_+(u_+)-\Delta(u_+).$$ 
Finally, we recall that $u_+$ is the global minimizer of $\F_+$ on $K_+$ and that $\max (u_+, u_-)\in K_+$, and so 
 $\F_+(u_+)\leq\F_+(\max(u_+, u_-))$.
 We conclude that $\Delta(\max(u_+,u_-))- \Delta(u_+)\geq 0$.
\end{proof}

\providecommand{\bysame}{\leavevmode\hbox to3em{\hrulefill}\thinspace}
\providecommand{\MR}{\relax\ifhmode\unskip\space\fi MR }
\providecommand{\MRhref}[2]{%
  \href{http://www.ams.org/mathscinet-getitem?mr=#1}{#2}
}
\providecommand{\href}[2]{#2}

\end{document}